\documentclass[12pt]{amsart}
\usepackage{textgreek}
\usepackage[english]{babel}
\usepackage{graphicx}
\usepackage{caption}
\usepackage{float}
\usepackage[normalem]{ulem}
\useunder{\uline}{\ul}{}
\usepackage{rotating}
\usepackage{gensymb}
\usepackage{tikz}
\usepackage{fancyhdr}
\usepackage{amssymb}
\usepackage{amsmath}
\usepackage{amsthm}
\usepackage{amsfonts}
\usepackage[all]{xy}
\usepackage{mathrsfs}
\usepackage{latexsym}
\usepackage{mathdots}
\usepackage{setspace}
\usepackage{hyperref}
\usepackage{booktabs}
\usepackage{longtable}
\usepackage{tabularx}
\usetikzlibrary{calc}
\hypersetup{
    colorlinks=true,
    linkcolor=blue,
    filecolor=blue,      
    urlcolor=blue,
    citecolor=red,
}
\usepackage{url}
\usepackage[utf8x]{inputenc}
\usepackage{graphicx}
\graphicspath{{images/}}
\usepackage{fancyhdr}
\usepackage[a4paper, margin=0.9in]{geometry}
\usepackage{tikz-cd}
\usepackage{tikz}
\usetikzlibrary{chains}
\usepackage{enumitem}
\usepackage{nicematrix}
\usepackage[title]{appendix}
\usetikzlibrary{positioning}
\tikzcdset{
  every arrow/.append style={
    shorten <=0pt,
    shorten >=0pt
  }
}
\tikzcdset{
  every label/.append style={font=\scriptsize\rmfamily}
}

\DeclareMathAlphabet{\mathcal}{OMS}{cmsy}{m}{n}
\DeclareFontFamilySubstitution{LGR}{\rmdefault}{cm}
\DeclareMathOperator{\Sym}{Sym}
\DeclareMathOperator{\Mor}{Mor}
\DeclareMathOperator{\Bl}{Bl}
\DeclareMathOperator{\Hilb}{Hilb}
\DeclareMathOperator{\ev}{ev}
\DeclareMathOperator{\sm}{sm}
\DeclareMathOperator{\Pic}{Pic}

\newtheorem{theorem}{Theorem}[section]
\newtheorem{lemma}[theorem]{Lemma}
\newtheorem{prop}[theorem]{Proposition}

\theoremstyle{definition}
\newtheorem{defn}[theorem]{Definition}

\theoremstyle{remark}
\newtheorem{remark}[theorem]{Remark}
\newcommand{\bb}{\mathbb{P}}
\newcommand{\iic}{\text{ impose independent conditions }}
\title{10}								
\makeatletter
\let\thetitle\@title
\let\theauthor\@author
\makeatother

\begin{document}
\hyphenpenalty=5000
\tolerance=1000
\title{Finiteness of the Chow ring of $\mathcal{M}_{5,10}$}

\author{Yuhan Liu}
\maketitle

\begin{abstract}
In this paper, we prove that the rational Chow ring $A^*(\mathcal{M}_{5,10})$ is finitely generated as a $\mathbb{Q}$-algebra.
\end{abstract}

%\text{MSC classes: 14C15, 14C17}

\section{Introduction}

For the past century, the moduli space of genus $g$ curves has occupied a central role in algebraic geometry. This space admits a natural compactification, $\overline{\mathcal{M}}_g$, which parametrizes stable curves. By passing to the normalization of these singular stable curves, one naturally recovers smooth curves equipped with marked points at the preimages of the nodes. Therefore, it is very important to understand the Chow ring of the moduli spaces $\mathcal{M}_{g,n}$, which paramatrize smooth genus $g$ curves with $n$ distinct marked points. 

One of the natural questions about moduli spaces $\mathcal{M}_{g,n}$ is to determine the Chow ring of $\mathcal{M}_{g,n}$. In \cite{M4} \cite{M5} \cite{M6} \cite{M789}, the authors have proved that the rational Chow ring of $A^*(\mathcal{M}_{g})$ is tautological for $4 \leq g \leq 9$. Canning and Larson proved $A^*(\mathcal{M}_{3,n})$ is tautological for $n \leq 11$; $A^*(\mathcal{M}_{4,n})$ is tautological for $n \leq 11$; $A^*(\mathcal{M}_{5,n})$ is tautological for $n \leq 7$; $A^*(\mathcal{M}_{6,n})$ is tautological for $n \leq 5$ \cite{HS}. Furthermore, in \cite{Liu}, the author proved that $A^*(\mathcal{M}_{5,n})$ is tautological for $n=8$ and $n=9$. In particular, the Chow rings in the above cases are finitely generated. In this paper, we prove:

\begin{theorem}\label{theorem}
    The Chow ring of $\mathcal{M}_{5,10}$ is finitely generated as a $\mathbb{Q}$-algebra.
\end{theorem}

Our main tool to prove it is the following theorem: 

\begin{theorem}\label{thm}

Let $C\subset\bb^4$ be a smooth, canonically embedded curve of genus $5$, neither hyperelliptic nor trigonal. Let $p_1,\ldots,p_{10}\in C$ be distinct points, and assume that any nine of them impose independent conditions on quadrics. If $p_1,\ldots,p_{10}$ fail to impose independent conditions on quadrics, then there exists a curve $E\subset\bb^4$ of degree $5$ and arithmetic genus $1$, having one of the configurations in Table \ref{tab:configurations}, interpreted according to the convention below that table, such that
\[
  p_1,\ldots,p_{10}\in E_{\mathrm{sm}}
  \qquad\text{and}\qquad
  \mathcal O_E(p_1+\cdots+p_{10})
  \cong \mathcal O_E(2).
\]

Conversely, let $E\subset\bb^4$ be such a curve, and let $p_1,\ldots,p_{10}\in E_{\mathrm{sm}}$ be distinct points satisfying
\[
  \mathcal O_E(p_1+\cdots+p_{10})
  \cong \mathcal O_E(2).
\]
Then these points fail to impose independent conditions on quadrics and lie on a smooth complete intersection of three quadrics in $\bb^4$.

\end{theorem}

\tikzset{
  ec component/.style={
    circle, draw,
    minimum size=5.5mm,
    inner sep=0pt,
    font=\small
  },
  ec singularity/.style={
    rectangle, draw,
    minimum size=5.5mm,
    inner sep=0pt,
    font=\small
  },
  ec degree/.style={
    font=\scriptsize,
    inner sep=1pt
  }
}

% Uniform bounding box for the diagrams.
\newenvironment{ecdiagram}{%
  \begin{tikzpicture}[
    x=1cm, y=1cm,
    line width=.45pt,
    every label/.style={ec degree}
  ]
  \path[use as bounding box]
    (-1.7,-1.45) rectangle (1.7,1.45);
}{%
  \end{tikzpicture}
}

% Smooth elliptic quintic.
\newcommand{\ecSmooth}{%
  \begin{ecdiagram}
    \node[
      ec component,
      label=below:{$\deg 5$}
    ] at (0,0) {$1$};
  \end{ecdiagram}
}

% Irreducible nodal rational quintic.
\newcommand{\ecNodal}{%
  \begin{ecdiagram}
    \node[
      ec component,
      label=below:{$\deg 5$}
    ] (a) at (0,-.45) {$0$};

    \draw (a.65)
      .. controls (.48,1.15) and (-.48,1.15)
      .. (a.115);
  \end{ecdiagram}
}

% Two components meeting at two nodes.
% Arguments: upper degree, lower degree.
\newcommand{\ecTwoCycle}[2]{%
  \begin{ecdiagram}
    \node[ec component,label=above:{$\deg #1$}]
      (a) at (0,.55) {$0$};
    \node[ec component,label=below:{$\deg #2$}]
      (b) at (0,-.55) {$0$};

    \draw (a) to[bend left=38] (b);
    \draw (a) to[bend right=38] (b);
  \end{ecdiagram}
}

% Triangle.
% Arguments: top, bottom-left, bottom-right degrees.
\newcommand{\ecTriangle}[3]{%
  \begin{ecdiagram}
    \node[ec component,label=above:{$\deg #1$}]
      (a) at (0,.65) {$0$};
    \node[ec component,label=below:{$\deg #2$}]
      (b) at (-.75,-.55) {$0$};
    \node[ec component,label=below:{$\deg #3$}]
      (c) at (.75,-.55) {$0$};

    \draw (a)--(b)--(c)--(a);
  \end{ecdiagram}
}

% Four-component cycle.
\newcommand{\ecFourCycle}{%
  \begin{ecdiagram}
    \node[ec component,label=above:{$\deg 1$}]
      (a) at (-.65,.55) {$0$};
    \node[ec component,label=above:{$\deg 1$}]
      (b) at (.65,.55) {$0$};
    \node[ec component,label=below:{$\deg 1$}]
      (c) at (.65,-.55) {$0$};
    \node[ec component,label=below:{$\deg 2$}]
      (d) at (-.65,-.55) {$0$};

    \draw (a)--(b)--(c)--(d)--(a);
  \end{ecdiagram}
}

% Five-component cycle.
\newcommand{\ecFiveCycle}{%
  \begin{ecdiagram}
    \node[ec component,label=above:{$\deg 1$}]
      (a) at (0,.8) {$0$};
    \node[ec component,label=left:{$\deg 1$}]
      (b) at (-.8,.2) {$0$};
    \node[ec component,label=below:{$\deg 1$}]
      (c) at (-.5,-.75) {$0$};
    \node[ec component,label=below:{$\deg 1$}]
      (d) at (.5,-.75) {$0$};
    \node[ec component,label=right:{$\deg 1$}]
      (e) at (.8,.2) {$0$};

    \draw (a)--(b)--(c)--(d)--(e)--(a);
  \end{ecdiagram}
}

% Elliptic m-fold configuration.
%
% #1: position of the degree-zero label on the square.
% #2: list of x/y/degree/label-position for the components.
%
% Example:
% \ecElliptic{left}{0/.85/2/above,0/-.85/3/below}
%
% Optional argument: vertical shift (default 0mm).
\newcommand{\ecElliptic}[3][0mm]{%
  \begin{ecdiagram}
    \begin{scope}[overlay, xshift=1.8mm, yshift=#1]
      \node[
        ec singularity,
        label=#2:{$\deg 0$}
      ] (s) at (0,0) {$1$};

      \foreach \xx/\yy/\dd/\pos [count=\i] in {#3} {
        \node[
          ec component,
          label=\pos:{$\deg \dd$}
        ] (b\i) at (\xx,\yy) {$0$};

        \draw (s)--(b\i);
      }
    \end{scope}
  \end{ecdiagram}
}

% Numbered table cell.
\newcommand{\ecCell}[2]{%
  \begin{minipage}[c]{.30\textwidth}
    \centering
    \resizebox{\linewidth}{!}{#2}\par
    \smallskip
    {\small (#1)}
  \end{minipage}%
}

%\end{theorem}

% No table environment, \clearpage, or \newpage here.
\begingroup
\setlength{\LTpre}{6pt}
\setlength{\LTpost}{6pt}
\setlength{\tabcolsep}{4pt}

\begin{longtable}{@{}ccc@{}}

\caption{The fifteen degree-five genus-one configurations.}
\label{tab:configurations}\\
\toprule
\endfirsthead

% Header on subsequent pages.
\multicolumn{3}{c}{%
  \small Table~\thetable\ (continued)
}\\
\toprule
\endhead

% Footer on all pages except the last.
\midrule
\multicolumn{3}{r}{\small Continued on the next page}\\
\endfoot

% No additional footer on the last page.
\endlastfoot

\ecCell{1}{\ecSmooth}
&
\ecCell{2}{\ecNodal}
&
\ecCell{3}{\ecTwoCycle{1}{4}}
\\[2mm]
\midrule

\ecCell{4}{\ecTwoCycle{2}{3}}
&
\ecCell{5}{\ecTriangle{1}{2}{2}}
&
\ecCell{6}{\ecTriangle{1}{1}{3}}
\\[2mm]
\midrule

\ecCell{7}{\ecFourCycle}
&
\ecCell{8}{\ecFiveCycle}
&
\ecCell{9}{%
  \ecElliptic[-2mm]{above}{0/-.85/5/below}%
}
\\[2mm]
\midrule

\ecCell{10}{%
  \ecElliptic{left}{%
    0/.85/1/above,
    0/-.85/4/below%
  }%
}
&
\ecCell{11}{%
  \ecElliptic{left}{%
    0/.85/2/above,
    0/-.85/3/below%
  }%
}
&
\ecCell{12}{%
  \ecElliptic[-2mm]{below}{%
    -1/0/1/below,
    0/.85/1/above,
    1/0/3/below%
  }%
}
\\[2mm]
\midrule

\ecCell{13}{%
  \ecElliptic[-2mm]{below}{%
    -1/0/1/below,
    0/.85/2/above,
    1/0/2/below%
  }%
}
&
\ecCell{14}{%
  \ecElliptic[-2mm]{below}{%
    -1/-.15/1/below,
    -.6/.8/1/above,
    .6/.8/1/above,
    1/-.15/2/below%
  }%
}
&
\ecCell{15}{%
  \ecElliptic{above}{%
    -1/0/1/below,
    -.7/.8/1/above,
    .7/.8/1/above,
    1/0/1/below,
    0/-.85/1/below%
  }%
}
\\[2mm]
\bottomrule

\multicolumn{3}{@{}p{.94\textwidth}@{}}{%
  \footnotesize
  \textit{Diagram Convention.}
  Circles record the genera of the component normalizations, with component degrees indicated beside them. In configurations (2)--(8), edges represent ordinary nodes, as in the usual dual graph. In configurations (9)--(15), the degree-zero square labeled
  $1$ denotes an elliptic $m$-fold point, where $m$ is its valency; its incident edges record the branches. The square denotes a singularity, not an additional component of the image. The local-ring descriptions of the elliptic $m$-fold points are given in the Appendix.
}\\

\end{longtable}
\endgroup

\textbf{Proof technique and outline of the paper.}
In the paper by Samir Canning and Hannah Larson \cite{HS}, they proved that all classes in $\mathcal{M}_{5,n}$ supported on the hyperelliptic or trigonal locus are tautological for $n \leq 12$. By excision, it suffices to show that the Chow ring of the open locus $\mathcal{M}_{5,n} \setminus \mathcal{M}_{5,n}^3$ in $\mathcal{M}_{5,n}$ is finitely generated for $n=10$, where $\mathcal{M}_{g,n}^k$ is the locus of curves of gonality $\leq k$. By corollary 2.9 in \cite{Liu}, we only need to focus on the locus where $10$ points fail to impose independent conditions.

\begin{defn}
    We denote the locus in $\mathcal{M}_{5,10} \setminus \mathcal{M}_{5,10}^3$ where 9 of the marked points fail to impose independent conditions on quadrics in $\bb^4$ by $\mathcal{M}_\omega$. 
\end{defn}
By Proposition 2.1 in \cite{Liu}, this is the locus where 8 of the marked points lie on a hyperplane.

\begin{defn}
    We denote the locus in $\mathcal{M}_{5,10} \setminus \mathcal{M}_{5,10}^3$ where any 9 marked points impose independent conditions and all 10 marked points fail to impose independent conditions on quadrics by $\mathcal{M}'$.
\end{defn}

In Section \ref{Section2}, we prove the finiteness of $A^*(\mathcal{M}_\omega)$, by realizing it as an open set of the Grassmann bundle over the configuration space of the marked points. For the rest of the paper, we are going to characterize the locus $\mathcal{M}'$ and prove its Chow ring is finitely generated. In Section \ref{scrolls}, we reduce to the cases that the 10 marked points lie on the specializations of smooth cubic scrolls. In Section \ref{Section4}, we furthermore reduce to the cases that the 10 marked points must lie on one of the configurations in Table \ref{tab:configurations}. In Section \ref{Sectionproof}, we give the necessary and sufficient conditions (Theorem \ref{thm}) for 10 marked points failed to impose independent conditions on quadrics. In Section \ref{Section6}, we prove that $A^*(\mathcal{M}')$ is finitely generated by constructing a proper surjective map from an open set of a Grassmann bundle over a moduli space whose Chow ring is finitely generated. We then prove Theorem \ref{theorem}. The appendix contains some clarifications and proofs of the facts we directly use in the paper.

\textbf{Notation.} Throughout the paper, we use $A^*(\cdot)$ to represent the Chow ring with $\textit{rational}$ coefficients. When we say a ring is finitely generated, we mean that it is finitely generated as a $\mathbb{Q}$-algebra; when we say a group is finitely generated, we mean that it is finitely generated as a $\mathbb{Q}$-vector space.

\textbf{Convention.} For any vector bundle $\mathcal{K}$, we define its projectivization $\mathbb{P}\mathcal{K}:=\mathscr{P}\mathrm{roj}(\Sym^\bullet \mathcal{K}^\vee)$. 

\textbf{Characteristic hypothesis.} We work over an algebraically closed field of characteristic not 2, 3 or 5.

\textbf{Acknowledgments.} I would like to thank my PhD advisor Eric Larson for helpful guidance throughout this project and feedback on earlier drafts of this manuscript. I would also like to thank ChatGPT for turning my hand drawn figures into beautiful tikz diagrams, and catching various minor typos throughout the paper.

\section{Finiteness of $A^*(\mathcal{M}_w)$}\label{Section2}

Recall that $\mathcal{M}_\omega$ is the dependent locus where 8 of the marked points lie on a hyperplane. Furthermore, any 9 marked points cannot lie on a hyperplane. Therefore, $\mathcal{M}_\omega$ is the disjoint union of $\{\mathcal{M}_{\omega,\{i,j\}}\}_{\{1 \leq i < j \leq 10\}}$, which is the locus where the points except $p_i$ and $p_j$ lie on a hyperplane. There are \[\binom{10}{2}=45\] of them. 

\begin{prop}\label{prop2.1}
    The Chow ring $A^*(\mathcal{M}_{\omega,\{i,j\}})$ is finitely generated for any $1 \leq i < j \leq 10$. The Chow ring $A^*(\mathcal{M}_\omega)$, as the disjoint intersection of $\{A^*(\mathcal{M}_{\omega,\{i,j\}})\}_{\{1 \leq i < j \leq 10\}}$, is also finitely generated.
\end{prop}
\textit{Proof}. The argument is similar to that of Section 4 in \cite{Liu}.  For the convenience of the reader, we briefly indicate the main steps. Without loss of generality, we prove that the Chow ring $A^*(\mathcal{M}_{\omega,\{9,10\}})$ is finitely generated.\\

Using the same notation as in Section 3 in \cite{Liu}, we have the exact sequence \begin{equation}
0 \rightarrow \mathcal{O}_{\mathcal{M}_{\omega,\{9,10\}}} \rightarrow \mathcal{L}^\vee \otimes f_*\omega_f \rightarrow \mathcal{G} \rightarrow 0,
\end{equation}
where $\mathcal{G}=\mathcal{L}^\vee \otimes \mathcal{G}'$. Take $G \subset PGL_5$ to be the stabilizer of the tuple $(H,p_9,p_{10})$, where $H$ is the hyperplane spanned by $p_1, \dots, p_8$. After a proper choice of coordinates, $G$ is the subgroup of $PGL_5$ consisting of matrices of the form 
\[
\left(
\begin{array}{c|c|c}
1 & 0 & \begin{matrix}0&0&0\end{matrix} \\
\hline
0 & 1 & \begin{matrix}*&*&*\end{matrix} \\
\hline
0 & 0 & \\
0 & 0 & \mathrm{GL}_3 \\
0 & 0 &
\end{array}
\right).
\]
We denote the universal bundle over $BG$ by $\mathcal{W}'$, and the subbundle $\mathcal{W} \subset \mathcal{W}'$ generated by the constant sections $(0,1,0,0,0),(0,0,1,0,0),(0,0,0,1,0),(0,0,0,0,1)$. Following the notations in \cite{Liu}, we further define $\sigma_10':BG \rightarrow \mathbb{P}\mathcal{W}'$ to be the fixed 10-th marked point. 
There, we have the composition maps
\[
\begin{tikzcd}
\mathcal{M}_{\omega,{\{9,10\}}}
\arrow[r,"b_9"]
&
(\mathbb{P}\mathcal{W})^7
\arrow[r,"\eta_i"]
\arrow[dr]
&
\mathbb{P}\mathcal{W}
\arrow[r,hookrightarrow,"i"]
\arrow[d,"\gamma"]
&
\mathbb{P}\mathcal{W}'
\arrow[dl, bend right=15, "\gamma'"'{yshift=-3.2pt,xshift=1.5pt}]
\\
&&
BG
\arrow[ur,bend right=25,"\sigma_9'"{pos=0.62, yshift=-3.2pt,xshift=1pt}]
\arrow[ur,bend right=45,"\sigma_{10}'"'{pos=0.55, yshift=3pt, xshift=-1.5pt}]
\end{tikzcd}
\]

Since the first seven marked points together with 9-th marked point and the 10-th marked point impose independent condition on quadrics in $\mathbb{P}^4$, the evaluation map is surjective and its kernel $\mathcal{E}$ is a vector bundle. Therefore, we have the exact sequence
\begin{equation}
0 \rightarrow \mathcal{E} \rightarrow \overline{\gamma}^*\gamma'_* \mathcal{O}_{\mathbb{P}\mathcal{W}'}(2) \xrightarrow{\text{evaluation map}}\bigoplus^7_{i=1}\eta_i^*i^*\mathcal{O}_{\mathbb{P}\mathcal{W}'}(2)\bigoplus^{10}_{k=9}\overline{\gamma}^*\sigma_k^*\mathcal{O}_{\mathbb{P}\mathcal{W}'}(2) \rightarrow 
0.
\end{equation}

We thus have the composition map $$\mathcal{M}_{\omega,\{9,10\}} \subset G(3, \mathcal{E}) \rightarrow U' \subset (\mathbb{P}\mathcal{W})^7 \rightarrow \mathbb{P}\mathcal{W} \rightarrow BG,$$
where each collection of points in the open set $U'$ is seven points which impose
independent condition on spaces of quadrics in $\mathbb{P}^3$. Therefore, $A^*(\mathcal{M}_{\omega,\{9,10\}})$ is finitely generated. The same argument applies for other locus, we thus finish the proof. \qed

\medskip
\noindent
\textbf{Standing assumption.}
Throughout the remainder of the paper, unless otherwise stated, we assume that every set of 9 points of $\{p_1,\ldots,p_{10}\}$ imposes independent conditions on quadrics in $\bb^4$.

\section{Nondegenerate specializations of smooth cubic scrolls}\label{scrolls}
%In this section, we aassume that any 9 points impose independent conditions on quadrics.

By \cite{Coskun} (example A1 in Section 5), we know that for 9 general points in $\mathbb{P}^4$, there exists a cubic scroll containing them. Note that this statement is only true for 9 points in general position, which doesn't always happen in our case. Therefore, we also need to consider the cases when the 9 points are not in general position. 
\begin{lemma}
    For any 9 distinct points in $\bb^4$, they lie on a flat limit of smooth cubic scrolls. 
\end{lemma}
\textit{Proof}. Note that the Hilbert polynomial $P(n)$ of a smooth cubic scroll is $\frac{3n^2+5n+2}{2}$. Let $\mathcal{H}$ be the Hilbert scheme $\Hilb^P(\bb^4)$, where $$P(n)=\frac{3n^2+5n+2}{2}.$$ Let $\mathcal{I} \subset \mathcal{H} \times (\bb^4)^9$ be the incidence locus $$\Bigl\{\left(S,q_1, \dots, q_9\right) \; \big| \; \text{S is a smooth cubic scroll and } q_i \in S \text{ for each } 1 \leq i \leq 9\Bigl\}.$$ Since the projection onto the second factor $p_2:\mathcal{I} \rightarrow (\bb^4)^9$ is proper, and its image contains a dense open, we conclude $p_2$ is surjective. That is, given any 9 distinct points in $\bb^4$, they must lie on a specialization of a smooth cubic scroll, i.e., a flat limit of a smooth cubic scroll. \qed

We study specializations of smooth cubic scrolls through their ruling maps. Note that a smooth cubic scroll can be viewed as an element in $\mathcal{M}_0(\mathbb{G}(1,4),3)$ and its specializations live in $\overline{\mathcal{M}}_0(\mathbb{G}(1,4),3)$, i.e., the collection of stable maps from a connected, at least nodal curve of genus 0 to $\mathbb{G}(1,4)$ of total degree 3. Note that the curve is reduced and each irreducible component must be isomorphic to $\bb^1$. Therefore, the only cases are as follows (via dual graph correspondence):

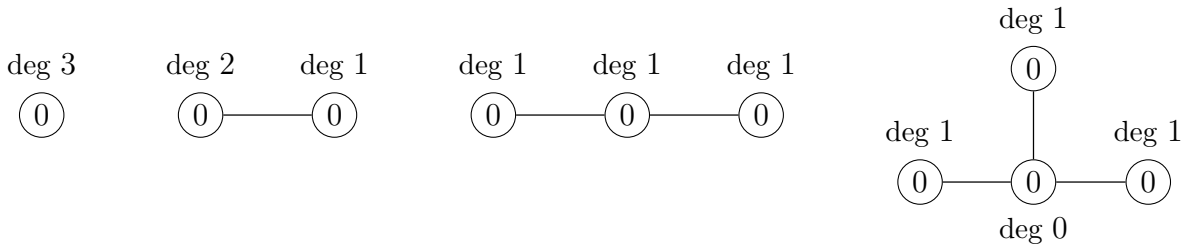
\[
\begin{tikzcd}[column sep=1.2em,row sep=1.2cm]
  |[draw,circle,inner sep=2.5pt,label=above:{\small\text{deg }3}]|0
\end{tikzcd}
\qquad
\begin{tikzcd}[column sep=1.2em,row sep=1.2cm]
  |[draw,circle,inner sep=2.5pt,label=above:{\small\text{deg }2}]|0
    \ar[-,from=1-1,to=1-2]
  &
  |[draw,circle,inner sep=2.5pt,label=above:{\small\text{deg }1}]|0
\end{tikzcd}
\qquad
\begin{tikzcd}[column sep=1.2em,row sep=1.2cm]
  |[draw,circle,inner sep=2.5pt,label=above:{\small\text{deg }1}]|0
    \ar[-,from=1-1,to=1-2]
  &
  |[draw,circle,inner sep=2.5pt,label=above:{\small\text{deg }1}]|0
    \ar[-,from=1-2,to=1-3]
  &
  |[draw,circle,inner sep=2.5pt,label=above:{\small\text{deg }1}]|0
\end{tikzcd}
\qquad
\begin{tikzcd}[column sep={1.5cm,between origins}, row sep={1.5cm,between origins}]
    & |[draw,circle,inner sep=2.5pt,label=above:{\small\text{deg }1}]|0 \\
  |[draw,circle,inner sep=2.5pt,label=above:{\small\text{deg }1}]|0
    \ar[-,from=2-1,to=2-2]
  &
  |[draw,circle,inner sep=2.5pt,label=below:{\small\text{deg }0}]|0
    \ar[-,from=2-2,to=2-3]
    \ar[-,from=2-2,to=1-2]
  & [-1.5em]
  |[draw,circle,inner sep=2.5pt,label=above:{\small\text{deg }1}]|0
\end{tikzcd}
\]

Let $f:X\to G(1,4)$ be the morphism corresponding to the quotient $\mathcal O_X^{\oplus 5}\twoheadrightarrow\mathcal E$.
Let
\[
\mathcal U
=
\{(\ell,p)\in \mathbb{G}(1,4)\times \bb^4 : p\in\ell\}
\]
be the incidence correspondence, and define
\[
Y:=X\times_{\mathbb{G}(1,4)}\mathcal U
\subset X\times \bb^4.
\]
In fact, we have $Y = \bb\mathcal{E}^\vee$, a $\bb^1$-bundle over $X$ whose fiber over $x\in X$ is the line represented by $f(x)$.
Let $\pi_2:Y\rightarrow \bb^4$ be the restriction of the second projection, and let $S\subset \bb^4$ be its image. Thus $S$ is the surface swept out by the lines parametrized by $f$. Our goal is to describe the resulting surfaces $S$ up to projective equivalence. Note that we only consider cases when $S$ is nondegenerate since they contain 9 marked points which impose independent conditions on quadrics.

Recall that $$\Mor\left(X, \mathbb{G}(1,4)\right)= \Bigl\{\mathcal{O}_X^{\oplus5} \twoheadrightarrow \mathcal{E}, \text{ where } \mathcal{E} \text{ is a vector bundle of rank 2 over } X\Bigl\},$$
so it suffices to consider all possible cases for $\mathcal{E}$.
Now we compute the image of the corresponding surface and its linear span over each component $\bb^1$ in each case. We use homogeneous coordinates $[x:y:z:w:r]$ on $\bb^4$. \\

\textbf{Case 1.} If the $\bb^1$ component has degree 3.\\
$\left(\mathrm{I}\right)$ $\mathcal E\simeq \mathcal O_{\bb^1}(1)\oplus\mathcal O_{\bb^1}(2).$\\

After a projective change of coordinates, we may take the basis of $H^0(\bb^1,\mathcal E)$ to be
\[
(s,0),\quad (t,0),\quad
(0,s^2),\quad (0,st),\quad (0,t^2),
\]
where $[s:t]$ are homogeneous coordinates on $\bb^1$.

For $x=[s_0:t_0]\in X$, the dual evaluation map $\mathcal E_x^\vee
\hookrightarrow H^0(X,\mathcal E)^\vee$ identifies $\mathcal E_x^\vee$ with the subspace
\[
\operatorname{Span}\bigl\{
(s_0,t_0,0,0,0),\,
(0,0,s_0^2,s_0t_0,t_0^2)
\bigr\}.
\]
Therefore, $f(x)\in\mathbb G(1,4)$ represents the line
\[
\ell_x=
\mathbf P\!\left(
\operatorname{Span}\bigl\{
(s_0,t_0,0,0,0),\,
(0,0,s_0^2,s_0t_0,t_0^2)
\bigr\}
\right)
\subset\bb^4.
\]
As $x$ varies in $X$, the homogeneous ideal of the surface is
\[
I =(zr-w^2,\ yz-xw,\ xr-yw).
\]
In particular, these lines sweep out the smooth cubic scroll. \\

$\left(\mathrm{II}\right)$ $\mathcal E\simeq \mathcal O_{\bb^1}\oplus\mathcal O_{\bb^1}(3).$\\
After a projective change of coordinates, we may take the basis of $H^0(\bb^1,\mathcal{E})$ to be
\[
(1,0),\quad (0,s^3),\quad
(0,s^2t),\quad (0,st^2),\quad (0,t^3),
\]
where $[s:t]$ are homogeneous coordinates on $\bb^1$.
Similar computation shows that $f(x) \in \mathbb{G}(1,4)$ represents the line 
\[
\ell_x=
\mathbf P\!\left(
\operatorname{Span}\bigl\{
(1,0,0,0,0),\,
(0,s_0^3,s_0^2t_0,t_0^2s_0,t_0^3)
\bigr\}
\right)
\subset\bb^4.
\]
As $x$ varies in $X$, the homogeneous ideal of the surface is
\[
I=(yr-zw,\ z^2-yw,\ w^2-zr).
\]
In particular, these lines sweep out a cone over the twisted cubic, which is a singular cubic scroll. \\

\textbf{Case 2.} If the $\bb^1$ component has degree 2.\\
$\left(\mathrm{I}\right)$ $\mathcal E\simeq \mathcal O_{\bb^1}(1)\oplus\mathcal O_{\bb^1}(1).$\\

After a projective change of coordinates, we may take the basis of $H^0(\bb^1,\mathcal{E})$ to be
\[
(s,0),\quad (t,0),\quad
(0,s),\quad (0,t)
\]
where $[s:t]$ are homogeneous coordinates on $\bb^1$.
Similar computation shows that $f(x) \in \mathbb{G}(1,4)$ represents the line 
\[
\ell_x=
\mathbf P\!\left(
\operatorname{Span}\bigl\{
(s_0,t_0,0,0),\,
(0,0,s_0,t_0)
\bigr\}
\right)
\subset\bb^3.
\]
As $x$ varies in $X$, the homogeneous ideal of the surface is
\[
I=(xw-yz).
\]
In particular, these lines sweep out a smooth quadric surface, whose linear span is a $\bb^3$. \\

$\left(\mathrm{II}\right)$ $\mathcal E\simeq \mathcal O_{\bb^1}\oplus\mathcal O_{\bb^1}(2).$\\

After a projective change of coordinates, we may take the basis of $H^0(\bb^1,\mathcal{E})$ to be
\[
(1,0),\quad (0,s^2),\quad
(0,st),\quad (0,t^2)
\]
where $[s:t]$ are homogeneous coordinates on $\bb^1$.
Similar computation shows that $f(x) \in \mathbb{G}(1,4)$ represents the line 
\[
\ell_x=
\mathbf P\!\left(
\operatorname{Span}\bigl\{
(1,0,0,0),\,
(0,s_0^2,s_0t_0,t_0^2)
\bigr\}
\right)
\subset\bb^3.
\]
As $x$ varies in $X$, the homogeneous ideal of the surface is
\[
I=(z^2-yw).
\]
In particular, these lines sweep out a singular quadric surface, whose linear span is a $\bb^3$. \\

\textbf{Case 3.} If the $\bb^1$ component has degree 1.\\
We have $\mathcal E\simeq \mathcal O_{\bb^1}\oplus\mathcal O_{\bb^1}(1).$\\
After a projective change of coordinates, we may take the basis of $H^0(\bb^1,\mathcal{E})$ to be
\[
(1,0),\quad (0,s),\quad
(0,t)
\]
where $[s:t]$ are homogeneous coordinates on $\bb^1$.
Similar computation shows that $f(x) \in \mathbb{G}(1,4)$ represents the line 
\[
\ell_x=
\mathbf P\!\left(
\operatorname{Span}\bigl\{
(1,0,0),\,
(0,s_0,t_0)
\bigr\}
\right)
\subset\bb^2.
\]
As $x$ varies in $X$, these lines sweep out a plane, whose linear span is a $\bb^2$. \\

\textbf{Case 3.} If the $\bb^1$ component has degree 0.\\
We have $\mathcal E\simeq \mathcal O_{\bb^1}\oplus\mathcal O_{\bb^1}.$\\ In this case $H^0(\mathcal{E})$ is generated by two constant sections, thus they sweep out a line, whose linear span is $\bb^1$.

\begin{remark}\label{uniqueness}
    For either fixed vector bundle $\mathcal{E}$, the construction uses its complete spaces of sections. Therefore, in each choice of $\mathcal{E}$, any two resulting surfaces differ only by a choice of basis for the sections in $H^0(\bb^1,\mathcal{E})$, which corresponds precisely to an automorphism of its linear span. That is to say, the image of the corresponding surface over $\bb^1$ is unique up to projective equivalence.
\end{remark}

We sum up all the results in Table \ref{tab:image} below.

\begin{table}[htbp]
\centering
\caption{Images associated to the componentwise splitting types.}
\label{tab:image}
\renewcommand{\arraystretch}{1.3}

\begin{tabularx}{\textwidth}{@{}cc>{\centering\arraybackslash}Xc@{}}
\toprule
Degree
& Splitting type of $\mathcal E|_{X_i}$
& Image of the corresponding surface over $X_i$
& Linear span \\
\midrule
$0$
& $\mathcal O\oplus\mathcal O$
& A line
& $\bb^1$ \\

$1$
& $\mathcal O\oplus\mathcal O(1)$
& A plane
& $\bb^2$ \\

$2$
& $\mathcal O(1)\oplus\mathcal O(1)$
& A smooth quadric surface
& $\bb^3$ \\

$2$
& $\mathcal O\oplus\mathcal O(2)$
& A quadric cone
& $\bb^3$ \\

$3$
& $\mathcal O(1)\oplus\mathcal O(2)$
& A smooth cubic scroll
& $\bb^4$ \\

$3$
& $\mathcal O\oplus\mathcal O(3)$
& A cone over a twisted cubic
& $\bb^4$ \\
\bottomrule
\end{tabularx}

\smallskip
\begin{minipage}{\textwidth}
\footnotesize
Here $X_i\simeq\bb^1$, and $\mathcal O(a)$ denotes $\mathcal O_{\bb^1}(a)$. The degree is $\deg(\mathcal E|_{X_i})$. 
\end{minipage}
\end{table}

\textbf{Case 1.} In this case we have $\mathcal{E}$ is a rank 2 vector bundle of degree 3 over $\bb^1$. 

$\left(\mathrm{I}\right)$ $\mathcal{E}=\mathcal{O}_X(1) \oplus \mathcal{O}_X(2)$. By Table \ref{tab:image}, we know that the resulting surface is a smooth cubic scroll. We denote by $S_1$.

$\left(\mathrm{II}\right)$ $\mathcal{E}=\mathcal{O}_X \oplus \mathcal{O}_X(3)$. By Table \ref{tab:image}, we know that the resulting surface is a singular cubic scroll. We denote by $S_2$.

\textbf{Uniqueness}. In these two cases, uniqueness follows from Remark \ref{uniqueness}.\\

\textbf{Case 2.} In this case we have $\mathcal{E}$ is a rank 2 vector bundle of total degree 3 over a union of two $\bb^1$'s. 

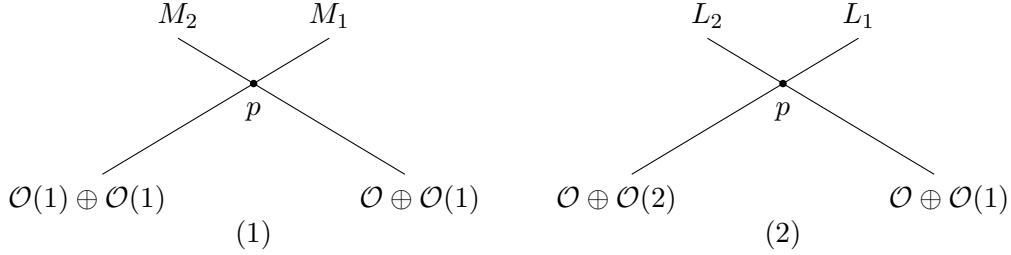
\[
\begin{tikzpicture}[every node/.style={font=\small}]

%========================
% Left graph
%========================
\coordinate (P2) at (0,0);

% X-shape (two straight lines crossing at P1)
\draw (-2,-1.2) -- (P2) -- (1,0.6);
\draw (-1, 0.6) -- (P2) -- (2,-1.2);

% intersection point p
\fill (P2) circle (1.4pt);
\node[below=3pt] at (P2) {$p$};

  \node[below]  at (-2.2,-1.2) {$\mathcal O(1)\oplus\mathcal O(1)$};
  \node[below] at (2.2,-1.2)  {$\mathcal O\oplus\mathcal O(1)$};
  \path (1,0.9) node {$M_1$};
  \path (-1,0.9) node {$M_2$};
  \path (0,-2) node {$\left(1\right)$};

% labels

%========================
% Right graph
%========================
\begin{scope}[xshift=7cm]
  \coordinate (P1) at (0,0);

  \draw (-2,-1.2) -- (P1) -- (1,0.6);
  \draw (-1, 0.6) -- (P1) -- (2,-1.2);

  \fill (P1) circle (1.4pt);
  \node[below=3pt] at (P1) {$p$};

\node[below]  at (-2.2,-1.2) {$\mathcal O\oplus\mathcal O(2)$};
\node[below] at (2.2,-1.2)  {$\mathcal O\oplus\mathcal O(1)$};
\path (1,0.9) node {$L_1$};
\path (-1,0.9) node {$L_2$};

\path (0,-2) node {$\left(2\right)$};

\end{scope}

\end{tikzpicture}
\]

$\left(\mathrm{I}\right)$ $\mathcal{E}|_{M_1} \cong \mathcal{O}(1) \oplus \mathcal{O}(1)$ and $\mathcal{E}|_{M_2} \cong \mathcal{O} \oplus \mathcal{O}(1)$. From Table \ref{tab:image}, $M_1$ sweep out a smooth quadric surface $Q$ and $M_2$ sweeps out a plane $P$. Since $M_1 \cap M_2 = p$, we know that $Q \cap P$ is a ruling of $Q$. Nondegeneracy excludes the case when $P$ is tangent to $Q$. In this case, the variety $S_3$ is a union of a smooth quadric surface and a plane, intersecting at one of the rulings.

\textbf{Uniqueness}. For two surfaces of this configuration, we may assume the quadric surface $Q$ are in the same position and they span $H \subset \bb^4$. Automorphism of $H$ acts transitively on its lines of $Q$ under the $Q \simeq \bb^1 \times \bb^1$. We may furthermore assume the attached ruling are the same. The groups of automorphism of $\bb^4$ fixes $H$ pointwise acts transitively on the planes attached to the quadric.

$\left(\mathrm{II}\right)$ $\mathcal{E}|_{L_1} \cong \mathcal{O} \oplus \mathcal{O}(2)$ and $\mathcal{E}|_{L_2} \cong \mathcal{O} \oplus \mathcal{O}(1)$. From Table \ref{tab:image}, $L_1$ sweeps out a singular quadric surface $Q$ and $M_2$ sweeps out a plane $P$. Since $L_1 \cap L_2 = p$, we know that $Q \cap P$ is a ruling of $Q$. Nondegeneracy excludes the case when $P$ is tangent to $Q$. In this case, the variety $S_4$ is a union of a singular quadric and a plane, intersecting at the ruling.

\textbf{Uniqueness}. The argument is analogous to the previous one. For two surfaces of this configuration, we may assume the singular quadric surface $Q \subset H$ and the attached ruling are in the same position since automorphisms preserving a quadric cone act transitively on its ruling. The group of automorphism of $\bb^4$ fixes $H$ pointwise acts transitively on the planes attached to the singular quadric. \\

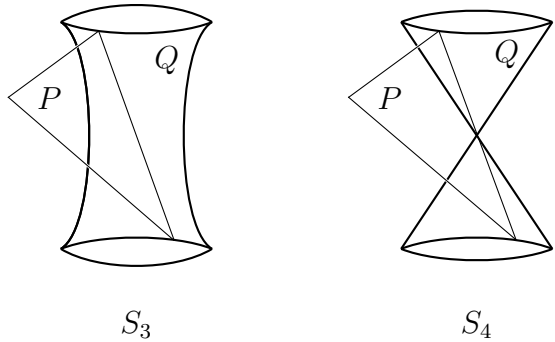
\[
\begin{tikzpicture}

% S_3
\begin{scope}[xshift=0cm, yshift=0cm, scale=1]

\draw[thick]
(0,0)
  .. controls (0.5,0.3) and (0.5,2.7) .. (0,3);

\draw[thick]
(0,0)
  .. controls (0.5,-0.3) and (1.5,-0.3) .. (2,0);

\draw[thick]
(0,0)
  .. controls (0.5,0.3) and (0.5,2.7) .. (0,3);

\draw[thick]
(2,0)
  .. controls (1.5,0.3) and (1.5,2.7) .. (2,3);

\draw[thick]
(0,3)
  .. controls (0.5,3.3) and (1.5,3.3) .. (2,3);

\draw[draw=white, line width=1pt]
(0.5,2.88) -- (-0.7,2.01);

\draw[draw=white, line width=1pt]
(-0.7,2.01) -- (1.5,0.11);

\draw[thick]
(0,3)
  .. controls (0.5,2.8) and (1.5,2.8) .. (2,3);

\draw[thick]
(0,0)
  .. controls (0.5,0.2) and (1.5,0.2) .. (2,0);

\draw[thin]
(0.5,2.88) -- (1.5,0.11);

\draw[thin]
(-0.7,2) -- (0.5,2.88);

\draw[thin]
(-0.7,2.01) -- (1.5,0.11);
\node at (-0.15,2.00) {$P$};
\node at (1.40,2.50) {$Q$};
\path (1,-1) node {$S_3$};

\end{scope}

% S_4
\begin{scope}[xshift=4.5cm, yshift=0cm, scale=1]

\draw[thick]
(0,0)
  .. controls (0.5,-0.2) and (1.5,-0.2) .. (2,0);

\draw[thick]
(0,3)
  .. controls (0.5,3.2) and (1.5,3.2) .. (2,3);

\draw[thick]
(0,3) -- (2,0);

\draw[thick]
(0,0) -- (2,3);

\draw[draw=white, line width=1pt]
(-0.7,2) -- (0.5,2.88);

\draw[draw=white, line width=1pt]
(-0.7,2.01) -- (1.52,0.12);

\draw[thick]
(0,0)
  .. controls (0.5,0.2) and (1.5,0.2) .. (2,0);

\draw[thick]
(0,3)
  .. controls (0.5,2.8) and (1.5,2.8) .. (2,3);

\draw[thin]
(0.5,2.88) -- (1.52,0.12);

\draw[thin]
(-0.7,2) -- (0.5,2.88);

\draw[thin]
(-0.7,2.01) -- (1.52,0.12);
\node at (-0.15,2.00) {$P$};
\node at (1.40,2.57) {$Q$};

\path (1,-1) node {$S_4$};

\end{scope}

\end{tikzpicture}
\]

\textbf{Case 3.} In this case we have $\mathcal{E}$ is a rank 2 vector bundle of total degree 3 over a chain of three $\bb^1$'s. Therefore, $\mathcal{E}|_{L_i} \simeq \mathcal{O} \oplus \mathcal{O}(1)$ for $1 \leq i \leq 3$. 

% preamble: \usepackage{tikz}

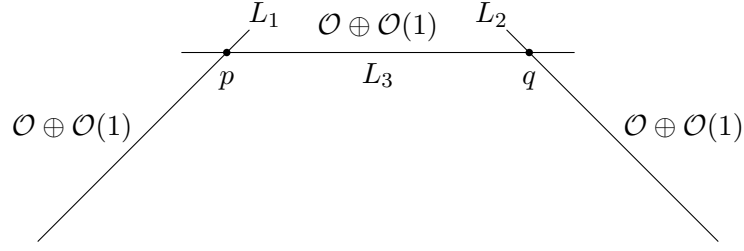
\[
\begin{tikzpicture}[every node/.style={font=\small}]

% intersection points
\coordinate (P) at (0,0);
\coordinate (Q) at (4,0);

% middle component (horizontal line), a chain link between P and Q
\draw (-0.6,0) -- (4.6,0);

% left component (slanted line through P)
\draw (0.3,0.3) -- (P) -- (-2.5,-2.5);

% right component (slanted line through Q)
\draw (3.7,0.3) -- (Q) -- (6.5, -2.5);

% mark points p and q
\fill (P) circle (1.4pt);
\fill (Q) circle (1.4pt);
\node[below=3pt] at (P) {$p$};
\node[below=3pt] at (Q) {$q$};
\path (0.5,0.5) node {$L_1$};
\path (3.5,0.5) node {$L_2$};
\path (2,-0.3) node {$L_3$};

\node[above] at (2,0) {$\mathcal O\oplus\mathcal O(1)$}; 
\node[left]  at (-1.1,-1.0) {$\mathcal O\oplus\mathcal O(1)$};   
\node[right] at (5.1,-1.0) {$\mathcal O\oplus\mathcal O(1)$};    

\end{tikzpicture}
\]

We find the variety $S_5$ by restricting $\mathcal{E}$ to each $\bb^1$. From Table \ref{tab:image}, each $L_i$ sweeps out a plane $P_i$. Since $p \neq q$, we have $P_1 \cap P_3$ and $P_2 \cap P_3$ meet at distinct lines $\ell_1$ and $\ell_2$ on $P_3$. Besides, $\ell_1$ and $\ell_2$ meet at a point. Thus $P_1, P_2, P_3$ meet at a point. In this case, the variety $S_5$ (showed below) is a union of three planes meeting at one point. 

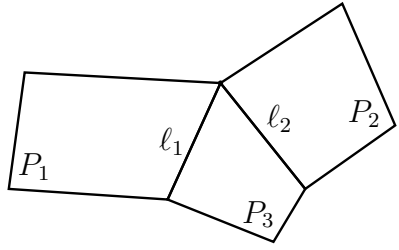
\[
\begin{tikzpicture}[line width=0.9pt, scale=0.7]
\coordinate (O) at (0,0);
\coordinate (P) at (-1,-2.2);   
\coordinate (Q) at (1.6,-2.0); 
\draw (O)--(P)--(1,-3.0)--(Q)--cycle;

\coordinate (P1) at (-4.0,-2.0);
\coordinate (P2) at (-3.7,0.2);
\draw (O)--(P)--(P1)--(P2)--cycle;
\coordinate (Q1) at (3.3,-0.8);
\coordinate (Q2) at (2.3,1.5);
\draw (O)--(Q)--(Q1)--(Q2)--cycle;
\fill (O) circle (1.4pt);
% Labels of the planes
\node[above right, inner sep=3pt] at (P1) {$P_1$};
\node[above left,  inner sep=3pt] at (3.2, -1.0) {$P_2$};
\node[above left,  inner sep=3pt] at (1.2,-2.9) {$P_3$};

% Labels of the intersection lines
\path (O) -- (P)
  node[midway, left, inner sep=3pt] {$\ell_1$};
\path (O) -- (Q)
  node[midway, above right, inner sep=1pt] {$\ell_2$};
\end{tikzpicture}
\]

\textbf{Uniqueness}. Note that for every such configuration, we can find basis $\{e_i\}_{1 \leq i \leq 5}$ such that $P_1 = \bb \langle e_1, e_2, e_4\rangle, P_2 = \bb\langle e_1, e_3, e_5\rangle$ and $P_3=\bb \langle e_1,e_2,e_3 \rangle$. Thus any two choices differ by an automorphism of $\bb^4$. \\

\textbf{Case 4.} In this case we have $\mathcal{E}$ is a rank 2 vector bundle of total degree 3 over a union of four $\bb^1$'s (shown below), where $L$ is contracted under $f$. Therefore, $\mathcal{E}|_{L_i} \simeq \mathcal{O} \oplus \mathcal{O}(1)$ for $1 \leq i \leq 3$. 

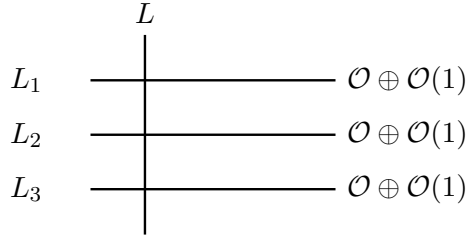
\[
\begin{tikzpicture}[every node/.style={font=\small}, line width=0.9pt, scale=0.6]

% Vertical line L
\draw (0,-2.2) -- (0,2.2);
\node[above] at (0,2.2) {$L$};

% Three horizontal lines (slightly tilted like your sketch)
\draw (-1.2, 1.2) -- (4.2, 1.2);
\draw (-1.2, 0.0) -- (4.2, 0.0);
\draw (-1.2,-1.2) -- (4.2,-1.2);

% Labels L1, L2, L3 on the left
\node[left] at (-2, 1.2) {$L_1$};
\node[left] at (-2, 0.0) {$L_2$};
\node[left] at (-2,-1.2) {$L_3$};

% Bundle labels on the right
\node[right] at (4.2, 1.2) {$\mathcal O\oplus\mathcal O(1)$};
\node[right] at (4.2, 0.0) {$\mathcal O\oplus\mathcal O(1)$};
\node[right] at (4.2,-1.2) {$\mathcal O\oplus\mathcal O(1)$};

\end{tikzpicture}
\]

We find the image by restricting $\mathcal{E}$ to each $\bb^1$. From Table $\ref{tab:image}$, each $L_i$ sweep out a plane $P_i$. Since $L$ is contracted by $f$, these $P_i$ meet along a line $\ell.$ In this case, the variety $S_6$ (shown below) is a union of three planes meeting a line.

% preamble: \usepackage{tikz}

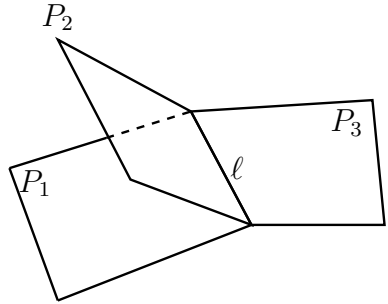
\[
\begin{tikzpicture}[line width=0.9pt, xscale=0.8, yscale=0.5]

\coordinate (A) at (-2.2,-5);
\coordinate (B) at (1,-3);
\coordinate (C) at (0,0);   
\coordinate (D) at (-3,-1.5);
\coordinate (I) at (-1.4,-0.7);

\coordinate (E) at (-2.2,1.9);  
\coordinate (F) at (-1,-1.8);  
\draw (C)--(E)--(F)--(B)--cycle;

\draw[dashed] (I) -- (C);
\draw (A) -- (D);
\draw (D) -- (I);
\draw (A) -- (B);
\coordinate (G) at (3.0,0.3);
\coordinate (H) at (3.2,-3);
\draw (C)--(G)--(H)--(B)--cycle;

% Labels of the planes
\node[below right, inner sep=3pt] at (-3,-1.3) {$P_1$};
\node[above,       inner sep=3pt] at (E) {$P_2$};
\node[below left,  inner sep=3pt] at (G) {$P_3$};

% Label of the common intersection line
\path (C) -- (B)
  node[midway, right, inner sep=3pt] {$\ell$};
\end{tikzpicture}
\]

\textbf{Uniqueness}. Note that for every such configuration, we can find basis $\{e_i\}_{1 \leq i \leq 5}$ such that $\ell = \bb \langle e_1, e_2\rangle, P_1 = \bb \langle e_1, e_2,e_3\rangle, P_2 = \bb\langle e_1, e_2,e_4\rangle$ and $P_3=\bb \langle e_1,e_2,e_5 \rangle$. Thus any two choices differ by an automorphism of $\bb^4$. \\
\begin{remark}
    Note that in all the cases above, the space of quadrics vanishing on each variety is always generated by three quadrics, i.e, $h^0(\mathcal{I}_{S_i}(2))=3$ for $1 \leq i \leq 6$. This property will be used again later in the proof.
\end{remark}

Below are all the nondegenerate specializations of the smooth cubic scroll, and a map from $S_i$ to $S_j$ means that $S_j$ is a specialization of $S_i$; see details in Table \ref{tab:cubic-surface-specializations} in Appendix \ref{A.2}. \\

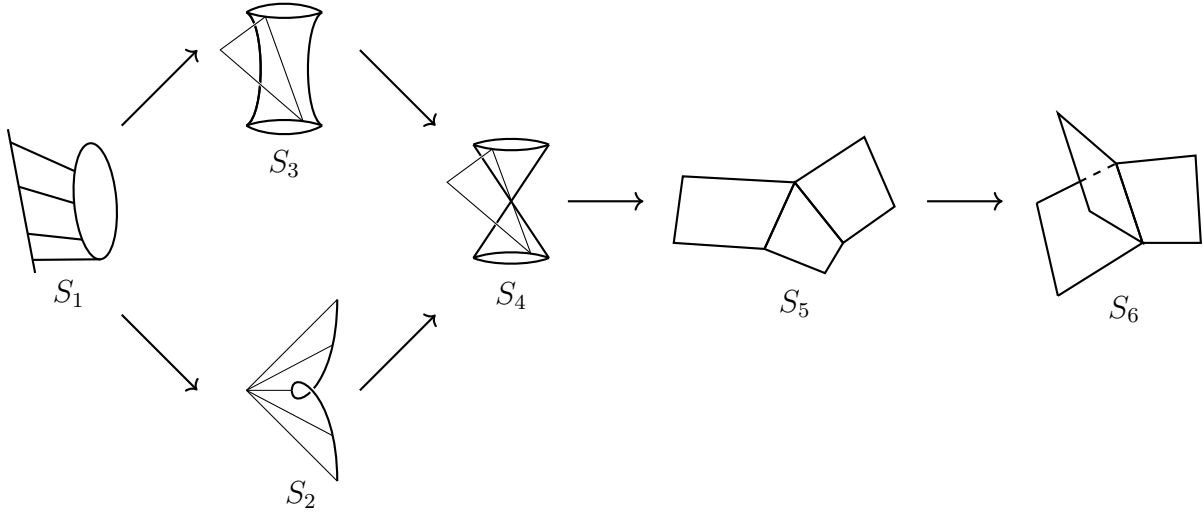
\begin{tikzpicture}[line width=0.8pt, scale=0.5]
\begin{scope}[yshift=-2cm, scale=0.7]

%-----------------------------------------------------------------------------------------------------------------------
%S_1
\begin{scope}[rotate=95]
\draw (0,0) ellipse (2.2 and 0.8);
\draw (-2.5,2.5) -- (3,3.05);
\coordinate (L1) at (-2,2.55);
\coordinate (E1) at (-2.2,0);
\draw (L1) -- (E1);
\coordinate (L2) at (-1, 2.65);
\coordinate (E2) at (-7/5,  {24*sqrt(2)/55});
\draw (L2) -- (E2); 
\coordinate (L3) at (0.8, 2.83);
\coordinate (E3) at (0,  0.8);
\draw (L3) -- (E3);
\coordinate (L4) at (2.5,3);
\coordinate (E4) at (6/5,  {4*sqrt(85)/55});
\draw (L4) -- (E4);
\end{scope}

\path (-1,-3.5) node {$S_1$};
%\path (-1,-3) node {smooth cubic scroll};
\end{scope}
\begin{scope}[xshift=-0.8cm]
\draw[->] (1.5,0) -- (3.5,2);
\end{scope}

\begin{scope}[xshift=-0.8cm]
\draw[->] (1.5,-5) -- (3.5,-7);
\end{scope}

%-----------------------------------------------------------------------------------------------------------------------
%S_2
\begin{scope}[xshift=4cm, yshift=-7cm, scale=1.2]

\draw
(1,0)
    .. controls (1,-0.5) and (2,0) .. (2,2);

\draw[draw=white, line width=2pt]
(1,0)
    .. controls (1, 0.5) and (2,0) .. (2,-2);
    
\draw
(1,0)
    .. controls (1, 0.5) and (2,0) .. (2,-2);
\draw[very thin]
(0,0) -- (1,0);
\draw[thin]
(0,0) -- (1.9,-1);
\draw[thin]
(0,0) -- (2,-2);
\draw[thin]
(0,0) -- (2,2);
\draw[very thin]
(0,0) -- (1.9,1);
\path(1.2, -2.3) node {$S_2$};
\end{scope}

%-----------------------------------------------------------------------------------------------------------------------
%S_3
\begin{scope}[xshift=4cm, yshift=0cm, scale=1]

\draw
(0,0)
    .. controls (0.5,0.3) and (0.5,2.7) .. (0,3);

\draw
(0,0)
    .. controls (0.5,-0.3) and (1.5,-0.3) .. (2,0);

\draw
(0,0)
    .. controls (0.5,0.3) and (0.5,2.7) .. (0,3);

\draw
(2,0)
    .. controls (1.5,0.3) and (1.5,2.7) .. (2,3);

\draw
(0,3)
    .. controls (0.5,3.3) and (1.5,3.3) .. (2,3);

\draw[draw=white, line width=1pt]
(0.5,2.88) -- (-0.7,2.01);

\draw[draw=white, line width=1pt]
(-0.7,2.01) -- (1.5, 0.11);

\draw
(0,3)
    .. controls (0.5,2.8) and (1.5,2.8) .. (2,3);

\draw
(0,0)
    .. controls (0.5,0.2) and (1.5,0.2) .. (2,0);
    
\draw[thin]
(0.5,2.88) -- (1.5, 0.11);

\draw[thin]
(-0.7,2) -- (0.5, 2.88);
\draw[thin]
(-0.7,2.01) -- (1.5, 0.11);
\path(1,-1) node {$S_3$};
\end{scope}

\begin{scope}[xshift=5.5cm, yshift=-7cm]
\draw[->] (1.5,0) -- (3.5,2);
\end{scope}

\begin{scope}[xshift=5.5cm, yshift=7cm]
\draw[->] (1.5,-5) -- (3.5,-7);    
\end{scope}

%-----------------------------------------------------------------------------------------------------------------------
%S_4
\begin{scope}[xshift=10cm, yshift=-3.5cm]
\draw
(0,0)
    .. controls (0.5,-0.2) and (1.5,-0.2) .. (2,0);  

\draw
(0,3)
    .. controls (0.5,3.2) and (1.5,3.2) .. (2,3);

\draw
(0,3) -- (2,0);
\draw
(0,0) -- (2,3);

\draw[draw=white, line width=1pt]
(-0.7,2) -- (0.5, 2.88);
\draw[draw=white, line width=1pt]
(-0.7,2.01) -- (1.52, 0.12);

\draw
(0,0)
    .. controls (0.5,0.2) and (1.5,0.2) .. (2,0);
    \draw
(0,3)
    .. controls (0.5,2.8) and (1.5,2.8) .. (2,3);

\draw[thin]
(0.5,2.88) -- (1.52, 0.12);

\draw[thin]
(-0.7,2) -- (0.5, 2.88);
\draw[thin]
(-0.7,2.01) -- (1.52, 0.12);
\path(1, -1) node {$S_4$};
\end{scope}

\begin{scope}[xshift=12.5cm, yshift=-2cm]
\draw[->] (0,0) -- (2,0);    
\end{scope}

%-----------------------------------------------------------------------------------------------------------------------
%S_5
\begin{scope}[xshift=18.5cm, yshift=-1.5cm, scale=0.8]
 \coordinate (O) at (0,0);
\coordinate (P) at (-1,-2.2);   
\coordinate (Q) at (1.6,-2.0); 
\draw (O)--(P)--(1,-3.0)--(Q)--cycle;

\coordinate (P1) at (-4.0,-2.0);
\coordinate (P2) at (-3.7,0.2);
\draw (O)--(P)--(P1)--(P2)--cycle;
\coordinate (Q1) at (3.3,-0.8);
\coordinate (Q2) at (2.3,1.5);
\draw (O)--(Q)--(Q1)--(Q2)--cycle;
%\fill (O) circle (1.4pt);   
\path (0, -4) node{$S_5$};
\end{scope}

\begin{scope}[xshift=22cm, yshift=-2cm]
\draw[->] (0,0) -- (2,0);    
\end{scope}

%-----------------------------------------------------------------------------------------------------------------------
%S_6
\begin{scope}[xshift=27cm, yshift=-1cm, scale=0.7]
\coordinate (A) at (-2.2,-5);
\coordinate (B) at (1,-3);
\coordinate (C) at (0,0);   
\coordinate (D) at (-3,-1.5);
\coordinate (I) at (-1.4,-0.7);

\coordinate (E) at (-2.2,1.9);  
\coordinate (F) at (-1,-1.8);  
\draw (C)--(E)--(F)--(B)--cycle;

\draw[dashed] (I) -- (C);
\draw (A) -- (D);
\draw (D) -- (I);
\draw (A) -- (B);
\coordinate (G) at (3.0,0.3);
\coordinate (H) at (3.2,-3);
\draw (C)--(G)--(H)--(B)--cycle;    
\path(0.3,-5.5) node{$S_6$};
\end{scope}
\end{tikzpicture}
%-------------------------------------------------------

\begin{prop}
Let $C$ be a smooth curve which is a complete intersection of three quadrics in $\bb^4$, and we have 10 points $p_1, \dots, p_{10}$ on $C$. For $1 \leq i \leq 6$, we assume $p_1, \dots, p_9 \in S_i$ and $p_{10} \notin S_i$. If $p_1, \dots, p_9 $ impose independent conditions on quadrics in $\bb^4$, then $p_1, \dots, p_{10} \iic$on quadrics in $\bb^4$.
\end{prop}
\begin{proof}
We first prove that $H^0(\mathcal{O}_{\bb^4}(2)) \rightarrow H^0(\mathcal{O}_{S_i}(2))$ is surjective for every $i$. Recall that for any $i$, we have $h^0(I_{S_i}(2))=3$, thus $h^0(\mathcal{O}_{S_i}(2)) \geq 12$, with equality if and only if the map is surjective. By upper semicontinuity of dimension of global sections, it suffices to check that $h^0(\mathcal{O}_{S_6}(2))=12$. $S_6$ is a union of three planes which intersect at a line, indeed the space of global quadric sections on it has dimension 12. 

Denote $\Gamma=\{p_1, \dots, p_9\}$. We have $0 \rightarrow \mathcal{I}_{S_i}(2) \rightarrow \mathcal{I}_{\Gamma}(2) \rightarrow \mathcal{I}_{\Gamma/S_i}(2) \rightarrow 0$, which induces long exact sequence in cohomology: $$0 \rightarrow H^0(\mathcal{I}_{S_i}(2)) \rightarrow H^0(\mathcal{I}_{\Gamma}(2)) \rightarrow H^0(\mathcal{I}_{\Gamma/S_i}(2)) \rightarrow H^1(\mathcal{I}_{S_i}(2)) \rightarrow \cdots.$$
Recall that we have $h^0(\mathcal{I}_{S_i}(2))=3$, and by our assumption of independence we have $h^0(\mathcal{I}_{\Gamma}(2))=15-9=6$. From the surjectivity of $H^0(\mathcal{O}_{\bb^4}(2)) \rightarrow H^0(\mathcal{O}_{S_i}(2))$, we obtain $h^1(\mathcal{I}_{S_i}(2))=0$. Therefore, we conclude that $h^0(\mathcal{I}_{\Gamma/S_i}(2))=6+0-3=3$.  

We have $0 \rightarrow \mathcal{I}_{{S_i} \cup p_{10}}(2) \rightarrow \mathcal{I}_{S_i}(2) \rightarrow \mathcal{O}_{p_{10}}(2) \rightarrow 0$, which induces long exact sequence in cohomology: $$0 \rightarrow H^0(\mathcal{I}_{{S_i} \cup p_{10}}(2)) \rightarrow H^0(\mathcal{I}_{S_i}(2)) \rightarrow H^0(\mathcal{O}_{p_{10}}(2)) \rightarrow H^1(\mathcal{I}_{{S_i} \cup p_{10}}(2)) \rightarrow 0.$$ Since $p_{10} \notin S_i$, we must have $h^0(\mathcal{I}_{{S_i} \cup p_{10}}(2))=2$. Therefore, $h^1(\mathcal{I}_{{S_i} \cup p_{10}}(2))=2+1-3=0$.

We also have $0 \rightarrow \mathcal{I}_{{S_i} \cup p_{10}}(2) \rightarrow \mathcal{I}_{{\Gamma} \cup p_{10}}(2) \rightarrow \mathcal{I}_{\Gamma/S_i}(2) \rightarrow 0$, which induces long exact sequence in cohomology: $$0 \rightarrow H^0(\mathcal{I}_{{S_i} \cup p_{10}}(2)) \rightarrow H^0(\mathcal{I}_{{\Gamma} \cup p_{10}}(2)) \rightarrow H^0(\mathcal{I}_{\Gamma/S_i}(2)) \rightarrow 0.$$ From this,  we conclude that $h^0(\mathcal{I}_{{\Gamma} \cup p_{10}}(2))=2+3-0=5$. That is, ${\Gamma} \cup p_{10} = \{p_1, p_2, \cdots p_{10}\} $  impose independent conditions on $\bb^4$. 
\end{proof}

\begin{remark}
    From above proposition, we reduce to the case where $p_1, p_2, \cdots p_{10}$ lie on one of our specializations $S_1, \cdots, S_6$. Furthermore, since $H^0(\mathcal{O}_{\bb^4}(2)) \rightarrow H^0(\mathcal{O}_{S_i}(2))$ is surjective for every $1 \leq i \leq 6$, 10 points impose independent conditions on quadrics on $S_i$ if and only if they impose independent conditions on quadrics in $\bb^4$.
\end{remark}

\section{Dependent locus}\label{Section4}

For each $1 \leq i \leq 6$, we prove the necessary conditions for $p_1, p_2, \cdots ,p_{10} \in S_i$ to fail to impose independent conditions on $\bb^4$. Note that these conditions already give some constraints on the configurations:

\begin{prop}
Under the assumption above, any 8 of the points can't lie on a hyperplane, any 3 of them can't be collinear and any 5 of them can't be coplanar.
\label{assumption}
\end{prop}
\begin{proof}
If 8 of the marked points lie on a hyperplane, they will fail to impose independent conditions on quadrics. In particular, any set of 9 points containing them will also fail to impose independent conditions on quadrics.

If 3 of the 10 points are collinear, i.e., lie on a line $L$, then any quadric vanishing on them will contain the line $L$. This forces our genus 5 curve containing $L$, contradicting to the fact that our curve is irreducible.

If 5 of the 10 points are coplanar, i.e., lie on a plane $P$, then any quadrics vanishing on them will contain a degree 2 curve $Q$ on $P$. This forces our genus 5 curve containing $Q$, contradicting to the fact that our curve is irreducible. 
\end{proof}

\begin{remark}
In the discussion below, we exclude the configurations corresponding to the cases described above.
\end{remark}

\subsection{$S_6$ : union of 3 planes, whose intersection is a line}
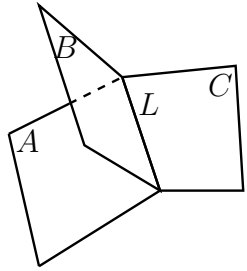
\[
\begin{tikzpicture}[line width=0.9pt, scale=0.5]

\coordinate (A) at (-2.2,-5);
\coordinate (B) at (1,-3);
\coordinate (C) at (0,0);   
\coordinate (D) at (-3,-1.5);
\coordinate (I) at (-1.4,-0.7);

\coordinate (E) at (-2.2,1.9);  
\coordinate (F) at (-1,-1.8); 
\draw (C)--(E)--(F)--(B)--cycle;
\draw[dashed] (I) -- (C);
\draw (A) -- (D);
\draw (D) -- (I);
\draw (A) -- (B);

\coordinate (G) at (3.0,0.3);
\coordinate (H) at (3.2,-3);
\draw (C)--(G)--(H)--(B)--cycle;
\node at (2.6,-0.2) {$C$};
\node at (0.7,-0.7) {$L$};
\node at (-1.5,0.8) {$B$};
\node at (-2.5,-1.7) {$A$};
\end{tikzpicture}
\]

We assume that there are $n$ points on the intersection line $L$, $a,b,c$ points on $A,B,C$ respectively. Therefore, we have 
\[
\left\{
\begin{aligned}
&a+b+c=10-n\\
&a+b<8-n,\\
& b+c<8-n,\\
& a+c <8-n,
\end{aligned}
\right.
\]
where $a,b,c,n \in \mathbb{Z}$.
The only cases in which a solution exists are $n=0$ and $n=1$.

For $n=0$, we have 4 points on one of $A,B$ and $C$ and 3 points on other two planes, all points are away from $L$; for $n=1$, we have 3 points on each plane away from $L$, and 1 point on the intersection $L$. In either case, we will have 7 points in $\bb^3$ and other 3 lie on a plane not in $\bb^3$. Therefore, in this configuration, we know that 10 points will impose independent conditions.

\subsection{$S_5$ : union of 3 planes, whose intersection is a point}

\begingroup

\def\CasePicture#1{%
\begin{tikzpicture}[
  line cap=round,
  line join=round,
  font=\small
]

% Same bounding box for both pictures.
\path[use as bounding box]
  (-0.15,-0.55) rectangle (7.35,4.65);

\coordinate (q) at (3.60,3.60);

% Foreground edges L_1 and L_2.
\draw[white,line width=2pt]
  (q) -- (3.00,1.00);
\draw[white,line width=2pt]
  (q) -- (5.40,1.20);

\draw[thick] (q) -- (3.00,1.00);
\draw[thick] (q) -- (5.40,1.20);

% Outside boundary of X.
\draw[thick]
  (3.00,1.00) -- (0.05,0.80)
  -- (0.65,2.95) -- (q);

% Outside boundary of Z.
\draw[thick]
  (q) -- (5.95,4.50)
  -- (7.15,2.15) -- (5.40,1.20);

% Lower boundary of Y.
\draw[thick]
  (3.00,1.00) -- (4.60,-0.40) -- (5.40,1.20);

% Four marked points on X.
\foreach \pt in {
  (1.60,2.45),(2.35,2.88),
  (1.45,1.90),(2.25,1.52)
}{
  \fill \pt circle (0.9pt);
}

% Six marked points on Z.
\foreach \pt in {
  (4.45,3.30),(5.05,3.65),(5.85,3.30),
  (4.95,2.80),(5.65,2.35),(6.30,2.40)
}{
  \fill \pt circle (0.9pt);
}

% ==================================================
% Case-dependent line, points, and labels.
% ==================================================
\ifnum#1=1\relax

  % Case 1: ell does not pass through q.
  \draw[thin]
    (3.18,1.78) -- (4.86,1.92);

  \fill (3.18,1.78) circle (1.1pt);
  \fill (4.02,1.85) circle (0.9pt);
  \fill (4.44,1.885) circle (0.9pt);

  \node at (3.46,3.87) {$q$};
  \node at (3.38,1.56) {$p$};
  \node at (4.53,1.67) {$\ell$};

\else

  % Case 2: ell passes through q=p.
  \coordinate (ellend) at (4.20,-0.05);

  \draw[thin]
    (q) -- (ellend)
    node[pos=0.57,right,inner sep=3pt] {$\ell$};

  \path (q) -- (ellend)
    coordinate[pos=0.42] (ellA)
    coordinate[pos=0.70] (ellB);

  \fill (ellA) circle (0.9pt);
  \fill (ellB) circle (0.9pt);

  \node at (3.46,3.95) {$q=p$};

\fi

% Vertex.
\fill (q) circle (1.4pt);

% Common labels.
\node at (0.90,2.82) {$X$};
\node at (1.43,2.76) {$\Gamma$};
\node at (4.57,-0.06) {$Y$};
\node at (5.74,4.13) {$Z$};
\node at (2.80,1.35) {$L_1$};
\node at (5.43,1.65) {$L_2$};

\end{tikzpicture}%
}

% ==================================================
% Two pictures on the same row.
% ==================================================
\begin{center}

\begin{minipage}[t]{0.48\linewidth}
  \centering
  \resizebox{\linewidth}{!}{\CasePicture{1}}
  \par\smallskip
  {\small Case~1}
\end{minipage}\hfill
\begin{minipage}[t]{0.48\linewidth}
  \centering
  \resizebox{\linewidth}{!}{\CasePicture{2}}
  \par\smallskip
  {\small Case~2}
\end{minipage}

\end{center}

\endgroup

We assume that there are $a$ points on $L_1$, $b$ points on $L_2$, and $c$ points on the vertex $q$; $x$ points on $X$ away from $L_1$, $y$ points on $Y$ away from $L_1, L_2$ and $z$ points on $Z$ away from $L_2$. Under our assumptions, we must have $2 \leq y \leq 4$.

If $y=4$, then we have $c=0, x+a=3$ and $b+z$=3. This is the configuration of  7 points in $\bb^3$ and $3$ points (not collinear) away from the $\bb^3$. Thus in the case 10 points impose independent conditions. 

If $y=3$, similar arguments will show that we have the same configuration as the case $y=4$. Thus in the case 10 points impose independent conditions. 

If $y=2$, there is only one possibility: $(c, x+a, b+z)=(0,4,4)$. We now analyze when it fails to impose independent conditions.

Denote the four points on $X$ away from $L_1$ by $\Gamma$, the line passing through the two points on $Y$ by $\ell$. We add two general points on $Z$, and figure out when these 12 points fail to impose independent conditions. If 10 points fail to impose independent conditions, then these 12 points must fail to impose independent conditions.

\textbf{Case 1. [$q \notin \ell$]:} 

\textit{Claim.} If $\Gamma \cup \{p,q\}$ don't lie on a plane conic or its specializations, then our 10 points impose independent conditions.

\textit{Proof of Claim.} Denote the space of quadrics vanish along these 12 points by $V$. Since the two points we add is general, any quadrics vanishing on the 6 points on $Z$ must vanish along $L_2$. Besides, elements in $V$ must vanish along $\ell$ since they vanish along 3 points on $\ell$. In particular, elements in $V$ also vanish along $p$ and $q$. If $\Gamma \cup \{p,q\}$ don't lie on a conic or its specializations, no quadrics on $S_5$ can vanish at these 12 points. Therefore, these 12 points impose independent conditions. In particular, the 10 points we have in the beginning impose independent condition too. This finishes the proof of our claim.

Here is the proof using exact sequence. We consider the exact sequence $$0 \rightarrow \mathcal{I}_{\Gamma \cup l \cup L_2}(2)|_{X \cup Y} \rightarrow \mathcal{I}_{\Gamma \cup l \cup L_2}(2)|_X \oplus \mathcal{I}_{\Gamma \cup l \cup L_2}(2)|_Y \rightarrow \mathcal{I}_{\{p,q\}}(2)|_{L_1} \rightarrow 0.$$

From this, we know that 
\begin{align*}
    h^0(\mathcal{I}_{\Gamma \cup \ell \cup L_2}(2)|_{X \cup Y}) &= h^0(\mathcal{I}_{\Gamma \cup \ell \cup L_2}(2)|_X) + h^0(\mathcal{I}_{\Gamma \cup \ell \cup L_2}(2)|_Y) - h^0(\mathcal{I}_{\{p,q\}}(2)|_{L_1}) \\
    &= h^0(\mathcal{I}_{\Gamma \cup \ell \cup L_2}(2)|_X) + 1 - 1 \\
    &= h^0(\mathcal{I}_{\Gamma \cup \ell \cup L_2}(2)|_X).
\end{align*}
If these 12 points fail to impose independent conditions, then we must have 
\begin{align*}
    h^0(\mathcal{I}_{\Gamma \cup \ell \cup Z}(2)|_{X \cup Y \cup Z}) &= h^0(\mathcal{I}_{\Gamma \cup \ell \cup L_2}(2)|_{X \cup Y})\neq 0.
\end{align*}

Therefore, we have $h^0(\mathcal{I}_{\Gamma \cup \ell \cup L_2}(2)|_X) \neq 0$, which will happen if and only if $\Gamma \cup \{p, q\}$ live on a plane conic.
By symmetry, 10 points fail to impose independent conditions only if one of the following situations occurs:

\begingroup

% Arguments:
% #1: left conic  — 0 = smooth, 1 = split
% #2: right conic — 0 = smooth, 1 = split
\def\ConicConfiguration#1#2{%
\begin{tikzpicture}[
  line cap=round,
  line join=round,
  font=\small
]

% Common bounding box.
\path[use as bounding box]
  (-0.45,-0.65) rectangle (7.55,4.85);

% Intersection points.
\coordinate (q) at (3.60,3.60);
\coordinate (p) at (3.18,1.78);
\coordinate (r) at (4.86,1.92);

% Intersections of the lines in the split cases.
\coordinate (TX) at (1.85,2.25);
\coordinate (TZ) at (5.80,3.20);

% Parameters for the smooth conic on X.
\pgfmathsetmacro{\conicXcx}{3.39-1.4*cos(70)}
\pgfmathsetmacro{\conicXcy}{2.69-0.65*cos(70)}
\pgfmathsetmacro{\conicXbx}{0.21/sin(70)}
\pgfmathsetmacro{\conicXby}{0.91/sin(70)}

% Parameters for the smooth conic on Z.
\pgfmathsetmacro{\conicZcx}{4.23+1.6*cos(70)}
\pgfmathsetmacro{\conicZcy}{2.76-0.1*cos(70)}
\pgfmathsetmacro{\conicZbx}{-0.63/sin(70)}
\pgfmathsetmacro{\conicZby}{0.84/sin(70)}

% Plane boundaries.
\draw[thick] (q) -- (3.00,1.00);
\draw[thick] (q) -- (5.40,1.20);

\draw[thick]
  (3.00,1.00) -- (0.05,0.80)
  -- (0.65,2.95) -- (q);

\draw[thick]
  (q) -- (5.95,4.50)
  -- (7.15,2.15) -- (5.40,1.20);

\draw[thick]
  (3.00,1.00) -- (4.60,-0.40) -- (5.40,1.20);

% The line ell.
\draw[thin,blue] (p) -- (r);

% ==================================================
% Left conic on X.
% ==================================================
\ifnum#1=0\relax

  % Smooth conic.
  \begin{scope}[
    cm={1.4,0.65,\conicXbx,\conicXby,
        (\conicXcx,\conicXcy)}
  ]
    % Outside the drawn plane: dashed.
    \draw[thin,blue,dash pattern=on 2pt off 2pt]
      (-70:1)
      arc[start angle=-70,end angle=70,radius=1];

    % Inside the plane: solid.
    \draw[thin,blue]
      (70:1)
      arc[start angle=70,end angle=290,radius=1];
  \end{scope}

  % Four marked points.
  \foreach \ang in {115,165,215,265}{
    \fill[red]
      ({\conicXcx+1.4*cos(\ang)+\conicXbx*sin(\ang)},
       {\conicXcy+0.65*cos(\ang)+\conicXby*sin(\ang)})
      circle (0.9pt);
  }

\else

  % Two lines through q and p, meeting at TX.
  \coordinate (Xa) at ($(q)!-0.18!(TX)$);
  \coordinate (Xb) at ($(q)!2.15!(TX)$);
  \coordinate (Xc) at ($(p)!-0.20!(TX)$);
  \coordinate (Xd) at ($(p)!2.10!(TX)$);

  % Dashed extensions.
  \draw[thin,blue,dash pattern=on 2pt off 2pt]
    (Xa) -- (Xb)
    (Xc) -- (Xd);

  % Solid portions inside X.
  \begin{scope}
    \clip
      (3.00,1.00) -- (0.05,0.80)
      -- (0.65,2.95) -- (q) -- cycle;
    \draw[thin,blue]
      (Xa) -- (Xb)
      (Xc) -- (Xd);
  \end{scope}

  % Two marked points on each line.
  \foreach \t in {0.40,1.40}{
    \fill[red] ($(q)!\t!(TX)$) circle (0.9pt);
    \fill[red] ($(p)!\t!(TX)$) circle (0.9pt);
  }

\fi

% ==================================================
% Right conic on Z.
% ==================================================
\ifnum#2=0\relax

  % Smooth conic.
  \begin{scope}[
    cm={-1.6,0.1,\conicZbx,\conicZby,
        (\conicZcx,\conicZcy)}
  ]
    % Outside the drawn plane: dashed.
    \draw[thin,blue,dash pattern=on 2pt off 2pt]
      (-70:1)
      arc[start angle=-70,end angle=70,radius=1];

    % Inside the plane: solid.
    \draw[thin,blue]
      (70:1)
      arc[start angle=70,end angle=290,radius=1];
  \end{scope}

  % Four marked points.
  \foreach \ang in {115,165,215,265}{
    \fill[red]
      ({\conicZcx-1.6*cos(\ang)+\conicZbx*sin(\ang)},
       {\conicZcy+0.1*cos(\ang)+\conicZby*sin(\ang)})
      circle (0.9pt);
  }

\else

  % Two lines through q and r, meeting at TZ.
  \coordinate (Za) at ($(q)!-0.15!(TZ)$);
  \coordinate (Zb) at ($(q)!1.55!(TZ)$);
  \coordinate (Zc) at ($(r)!-0.18!(TZ)$);
  \coordinate (Zd) at ($(r)!1.65!(TZ)$);

  % Dashed extensions.
  \draw[thin,blue,dash pattern=on 2pt off 2pt]
    (Za) -- (Zb)
    (Zc) -- (Zd);

  % Solid portions inside Z.
  \begin{scope}
    \clip
      (q) -- (5.95,4.50)
      -- (7.15,2.15) -- (5.40,1.20) -- cycle;
    \draw[thin,blue]
      (Za) -- (Zb)
      (Zc) -- (Zd);
  \end{scope}

  % Two marked points on each line.
  \foreach \t in {0.40,1.20}{
    \fill[red] ($(q)!\t!(TZ)$) circle (0.9pt);
    \fill[red] ($(r)!\t!(TZ)$) circle (0.9pt);
  }

\fi

% Intersection points.
\fill (q) circle (1.4pt);
\fill (p) circle (1.1pt);
\fill (r) circle (1.1pt);

% Two marked points on ell.
\fill[red] (4.02,1.85) circle (0.9pt);
\fill[red] (4.44,1.885) circle (0.9pt);

% Labels.
\node at (0.90,2.82) {$X$};
\node at (1.43,2.76) {$\Gamma$};
\node at (4.57,-0.06) {$Y$};
\node at (5.74,4.13) {$Z$};
\node at (3.46,3.87) {$q$};
\node at (3.38,1.56) {$p$};
\node at (2.87,1.35) {$L_1$};
\node at (5.40,1.65) {$L_2$};
\node at (4.25,1.60) {$\ell$};

\end{tikzpicture}%
}

% ==================================================
% Display the four diagrams, two per row.
% ==================================================
\begin{center}

% First row.
\begin{minipage}[t]{0.48\linewidth}
  \centering
  \resizebox{\linewidth}{!}{\ConicConfiguration{0}{0}}
  \par\smallskip
  {\small Configuration (5) in Table \ref{tab:configurations} }
\end{minipage}\hfill
\begin{minipage}[t]{0.48\linewidth}
  \centering
  \resizebox{\linewidth}{!}{\ConicConfiguration{1}{0}}
  \par\smallskip
  {\small Configuration (7) in Table \ref{tab:configurations}}
\end{minipage}

\par\vspace{10pt}

% Second row.
\begin{minipage}[t]{0.48\linewidth}
  \centering
  \resizebox{\linewidth}{!}{\ConicConfiguration{0}{1}}
  \par\smallskip
  {\small Configuration (7) in Table \ref{tab:configurations}}
\end{minipage}\hfill
\begin{minipage}[t]{0.48\linewidth}
  \centering
  \resizebox{\linewidth}{!}{\ConicConfiguration{1}{1}}
  \par\smallskip
  {\small Configuration (8) in Table \ref{tab:configurations}}
\end{minipage}

\end{center}
\endgroup

\textbf{Case 2. [$q \in \ell$]:}
Similarly, if $\Gamma \cup {q}$ don't lie on a conic or its specialization which vanishes of order 2 at $q$ along $L_1$, then our 10 original points impose independent conditions. Indeed, quadrics vanish along the 6 points on $Z$ must vanish along $\ell$ and $L_2$. Therefore, any quadrics vanish at these 12 points should vanish of order 2 when restricting to $L_1$. Below is the proof using exact sequences:
If these 12 points fail to impose independent conditions, then we must have $h^0(\mathcal{I}_{\Gamma \cup \ell \cup Z}(2)|_{X \cup Y \cup Z})=h^0(\mathcal{I}_{\Gamma \cup \ell \cup L_2}(2)|_{X \cup Y}) \neq 0$. In this case, our exact sequence is $$0 \rightarrow \mathcal{I}_{\Gamma \cup \ell \cup L_2}(2)|_{X \cup Y} \rightarrow \mathcal{I}_{\Gamma \cup \ell \cup L_2}(2)|_X \oplus \mathcal{I}_{\Gamma \cup \ell \cup L_2}(2)|_Y \rightarrow \mathcal{I}_{2p}(2)|_{L_1} \rightarrow 0.$$ 

From this, we know that 
\begin{align*}
    h^0(\mathcal{I}_{\Gamma \cup \ell \cup L_2}(2)|_{X \cup Y}) &= h^0(\mathcal{I}_{\Gamma \cup \ell \cup L_2}(2)|_X) + h^0(\mathcal{I}_{\Gamma \cup \ell \cup L_2}(2)|_Y) - h^0(\mathcal{I}_{2p}(2)|_{L_1}) \\
    &= h^0(\mathcal{I}_{\Gamma \cup \ell \cup L_2}(2)|_X) + 1 - 1 \\
    &= h^0(\mathcal{I}_{\Gamma \cup \ell \cup L_2}(2)|_X).
\end{align*}
Moreover, we have $h^0(\mathcal{I}_{\Gamma \cup \ell \cup L_2}(2)|_X) \neq 0$ if and only if there exists a plane conic, passing through $\Gamma \cup p$, tangent to $L_1$. By symmetry, 10 points fail to impose independent conditions only if one of the following situations occurs:
\begingroup

% First argument: conic on X.
% Second argument: conic on Z.
% 0 = smooth conic; 1 = two lines through p.
\def\TangentConicCase#1#2{%
\begin{tikzpicture}[
  line cap=round,
  line join=round,
  font=\small
]
% Same bounding box for all four pictures.
\path[use as bounding box]
  (-0.45,-0.65) rectangle (7.55,4.85);

\coordinate (p) at (3.60,3.60);
\coordinate (ellend) at (4.20,-0.05);

% --------------------------------------------------
% The three planes
% --------------------------------------------------

% L_1 and L_2.
\draw[thick] (p)--(3.00,1.00);
\draw[thick] (p)--(5.40,1.20);

% Outside boundary of X.
\draw[thick]
  (3.00,1.00)--(0.05,0.80)
  --(0.65,2.95)--(p);

% Outside boundary of Z.
\draw[thick]
  (p)--(5.95,4.50)
  --(7.15,2.15)--(5.40,1.20);

% Lower boundary of Y.
\draw[thick]
  (3.00,1.00)--(4.60,-0.40)--(5.40,1.20);

% ell passes through p, with two marked points.
\draw[thin,blue] (p)--(ellend)
  node[pos=0.57,right,inner sep=3pt] {$\ell$};

\foreach \t in {0.42,0.70}{
  \fill[red] ($(p)!\t!(ellend)$) circle (0.9pt);
}

% --------------------------------------------------
% Conic on X
% --------------------------------------------------

\ifnum#1=0\relax

  % Smooth ellipse through p, tangent to L_1 at p.
  % Draw the full ellipse dashed first.
  \begin{scope}[
    cm={1,0.65,-0.24,-1.04,(2.60,2.95)}
  ]
    \draw[thin,blue,dashed] (0,0) circle[radius=1];
  \end{scope}

  % Redraw the portion inside X as solid.
  \begin{scope}
    \clip
      (3.00,1.00)--(0.05,0.80)
      --(0.65,2.95)--(p)--cycle;

    \begin{scope}[
      cm={1,0.65,-0.24,-1.04,(2.60,2.95)}
    ]
      \draw[thin,blue] (0,0) circle[radius=1];
    \end{scope}
  \end{scope}

  % Four marked points on the smooth conic.
  \foreach \ang in {45,90,135,180}{
    \fill[red]
      ({2.60+cos(\ang)-0.24*sin(\ang)},
       {2.95+0.65*cos(\ang)-1.04*sin(\ang)})
      circle (0.9pt);
  }

\else

  % Two distinct lines through p.
  \coordinate (Xone) at (1.00,2.50);
  \coordinate (Xtwo) at (1.20,1.20);

  \coordinate (Xa) at ($(p)!-0.12!(Xone)$);
  \coordinate (Xb) at ($(p)!1.25!(Xone)$);
  \coordinate (Xc) at ($(p)!-0.12!(Xtwo)$);
  \coordinate (Xd) at ($(p)!1.18!(Xtwo)$);

  % Dashed extensions outside X.
  \draw[thin,blue,dashed]
    (Xa)--(Xb)
    (Xc)--(Xd);

  % Solid portions inside X.
  \begin{scope}
    \clip
      (3.00,1.00)--(0.05,0.80)
      --(0.65,2.95)--(p)--cycle;

    \draw[thin,blue]
      (Xa)--(Xb)
      (Xc)--(Xd);
  \end{scope}

  % Two marked points on each line.
  \foreach \t in {0.35,0.72}{
    \fill[red] ($(p)!\t!(Xone)$) circle (0.9pt);
    \fill[red] ($(p)!\t!(Xtwo)$) circle (0.9pt);
  }

\fi

% --------------------------------------------------
% Conic on Z
% --------------------------------------------------

\ifnum#2=0\relax

  % Smooth ellipse through p, tangent to L_2 at p.
  % Draw the full ellipse dashed first.
  \begin{scope}[
    cm={-1.25,0.1,0.6,-0.8,(4.85,3.50)}
  ]
    \draw[thin,blue,dashed] (0,0) circle[radius=1];
  \end{scope}

  % Redraw the portion inside Z as solid.
  \begin{scope}
    \clip
      (p)--(5.95,4.50)
      --(7.15,2.15)--(5.40,1.20)--cycle;

    \begin{scope}[
      cm={-1.25,0.1,0.6,-0.8,(4.85,3.50)}
    ]
      \draw[thin,blue] (0,0) circle[radius=1];
    \end{scope}
  \end{scope}

  % Four marked points on the smooth conic.
  \foreach \ang in {45,90,135,180}{
    \fill[red]
      ({4.85-1.25*cos(\ang)+0.6*sin(\ang)},
       {3.50+0.1*cos(\ang)-0.8*sin(\ang)})
      circle (0.9pt);
  }

\else

  % Two distinct lines through p.
  \coordinate (Zone) at (5.90,3.90);
  \coordinate (Ztwo) at (6.00,2.35);

  \coordinate (Za) at ($(p)!-0.12!(Zone)$);
  \coordinate (Zb) at ($(p)!1.35!(Zone)$);
  \coordinate (Zc) at ($(p)!-0.12!(Ztwo)$);
  \coordinate (Zd) at ($(p)!1.40!(Ztwo)$);

  % Dashed extensions outside Z.
  \draw[thin,blue,dashed]
    (Za)--(Zb)
    (Zc)--(Zd);

  % Solid portions inside Z.
  \begin{scope}
    \clip
      (p)--(5.95,4.50)
      --(7.15,2.15)--(5.40,1.20)--cycle;

    \draw[thin,blue]
      (Za)--(Zb)
      (Zc)--(Zd);
  \end{scope}

  % Two marked points on each line.
  \foreach \t in {0.35,0.72}{
    \fill[red] ($(p)!\t!(Zone)$) circle (0.9pt);
    \fill[red] ($(p)!\t!(Ztwo)$) circle (0.9pt);
  }

\fi

% --------------------------------------------------
% Vertex and labels
% --------------------------------------------------

\fill (p) circle (1.4pt);

\node at (0.90,2.82) {$X$};
\node at (1.30,2.35) {$\Gamma$};
\node at (4.57,-0.06) {$Y$};
\node at (5.74,4.13) {$Z$};

\node at (3.40,3.80) {$p$};
\node at (2.87,1.35) {$L_1$};
\node at (5.40,1.65) {$L_2$};

\end{tikzpicture}%
}

% ==================================================
% Arrange the four pictures: two in each row
% ==================================================

\begin{center}

% First row.
\begin{minipage}[t]{0.48\linewidth}
  \centering
  \resizebox{\linewidth}{!}{\TangentConicCase{0}{0}}
  \par\smallskip
  {\small Configuration (13) in Table \ref{tab:configurations}}
\end{minipage}\hfill
\begin{minipage}[t]{0.48\linewidth}
  \centering
  \resizebox{\linewidth}{!}{\TangentConicCase{0}{1}}
  \par\smallskip
  {\small Configuration (14) in Table \ref{tab:configurations}}
\end{minipage}

\par\vspace{12pt}

% Second row.
\begin{minipage}[t]{0.48\linewidth}
  \centering
  \resizebox{\linewidth}{!}{\TangentConicCase{1}{0}}
  \par\smallskip
  {\small Configuration (14) in Table \ref{tab:configurations}}
\end{minipage}\hfill
\begin{minipage}[t]{0.48\linewidth}
  \centering
  \resizebox{\linewidth}{!}{\TangentConicCase{1}{1}}
  \par\smallskip
  {\small Configuration (15) in Table \ref{tab:configurations}}
\end{minipage}

\end{center}
\endgroup

\subsection{$S_4$ : union of a plane and a singular quadric surface}
Using similar analysis, the possible cases are: 

$\left(\mathrm{I}\right)$ Four points on $P$ away from $L$ and six points on the singular quadric $Q$.\\
$\left(\mathrm{II}\right)$ Three points on $P$ away from $L$ and seven points on the singular quadric $Q$. \\
$\left(\mathrm{III}\right)$ Three points on $P$ away from $L$, the cone point $p$, and six points on the singular quadric $Q$.

For $\left(\mathrm{II}\right)$ and $\left(\mathrm{III}\right)$, we will have 7 points in $\bb^3$ and other 3 lie on a plane not in $\bb^3$, in which case 10 points impose independent conditions.

Now let's focus on the case $\left(\mathrm{I}\right)$. By blowing up $Q$ via $p$, we have the following diagram: 
\[\begin{tikzcd}
    \mathbb{F}_2 \arrow[r, hookrightarrow, "|e+2f|"] \arrow[d,"b"] &\Bl_p\bb^4 \arrow[d] \\
    Q \arrow[r] & \bb^4
\end{tikzcd}\]

Add 2 general points to $P$, and we determine the conditions under which these 12 points fail to impose independent conditions. Take any curve $E$ of class $e+3f$ in $\mathbb{F}_2$, and we now compute its genus and degree. By adjunction, we have $$2p_a(E)-2=E(E+K_{\mathbb{F}_2}).$$ Therefore, $p_a(E)=\frac{1}{2}((e+3f)(e+3f-2e-4f)+2)=0$. Since the hyperplane class is given by $e+2f$, we have $\deg E= (e+3f)(e+2f)=3$. Thus $E$ is a degree 3 curve with genus 0.

\begin{prop}
    If preimages of the 6 points $p_1, \cdots, p_6$ on $Q$ under blowup do not all lie on any member of $|e+3f|$, then 12 points impose independent conditions on quadrics on $S_4$. In particular, the original 10 points also impose independent conditions.
\end{prop}
\begin{proof}
These 12 points impose independent condition if and only if there are no quadrics on $S_4$ vanishing at these 12 points. Note that all quadrics on $S_4$ vanishing at the 6 points on $P$ must vanish along $P$. Without loss of generality, we assume the $S_4=V(yr-x^2,wy,wx)$. Moreover, the space of quadrics on $S_4$ vanishing on $P$ is generated by $x^2, xy, xz, xr, y^2, yz, xw, yw$. We want to see how many conditions the 6 points on $Q$ impose on this space. Note that all points on $Q$ satisfy $w=0$, so they impose no conditions on $xw$ and $yw$. We turn to consider the space $V$, which is generated by $x^2, xy, xz, xr, y^2, yz$. These sections vanish along $L$, which pulls back to the class $e+f$ on $\mathbb{F}_2$. These sections pull back to $H^0(\mathcal{O}_{\mathbb{F}_2}(2e+4f))$, thus they also lie in $H^0(\mathcal{O}_{\mathbb{F}_2}(e+3f))$. Since they both have dimension 6, we know $V=H^0(\mathcal{O}_{\mathbb{F}_2}(e+3f))$. Now we consider the evaluation map
$$H^0(\mathcal{O}_{\mathbb{F}_2}(e+3f)) \rightarrow H^0(\mathcal{O}_{\{p_1, \dots, p_6\}}(e+3f)).$$
If the 12 points fail to impose independent condition on quadrics on $S_4$, then this map will have a non trivial kernel. Choose a nonzero element $s$ of this kernel. The vanishing locus of $s$ is exactly $E$. Therefore, $p_1, \dots, p_6$ must lie on the curve $E$. \end{proof}

From now on, we are going to assume the 6 points $p_1, \cdots, p_6$ on $Q$ lie on our curve $E$. Note that $E$ might be reducible, so there are several possibilities. From our assumption for the configuration, we can't have nonreduced structure on E. Therefore, we only need to list all possible candidates for $E$ which are reduced:

$\left(\mathrm{I}\right)$ $E$ is irreducible with class $e+3f$. \\
$\left(\mathrm{II}\right)$ $E$ is the union of two irreducible curves $e+2f$ and $f$. \\
$\left(\mathrm{III}\right)$ $E$ is the union of four irreducible curves $e,f,f$ and $f$.

\begin{remark}
    Note that we have $(e+f)e=-1<0$, which implies that $e+f$ can't be irreducible. This is why $e+f$ doesn't appear as a component of $E$ in the above candidates.
\end{remark}
Denote $\Gamma=\{p_1, \cdots, p_{10}\}$, $\Gamma_2=\{p_1, \cdots, p_6\}$ and $\Gamma_1=\Gamma \setminus \Gamma_2$.

$\left(\mathrm{I}\right)$ The image of $E$ under $b$ is a rational curve $X$ of degree 3. 

Note that $(e+3f)e=1$ and $(e+3f)(e+f)=2$, we know that $X$ intersects $L$ at 2 points, denoted by $p$ and $q$. $p$ and $q$ might be the same, so we divide into 2 cases.

\textit{Case 1}: $p \neq q$. We have the following exact sequence: $$0 \rightarrow \mathcal{I}_{X \cup \Gamma_1}(2) \rightarrow \mathcal{I}_{X \cup \Gamma_1}(2)|_P \oplus \mathcal{I}_{X \cup \Gamma_1}(2)|_Q \rightarrow \mathcal{I}_{X \cup \Gamma_1}(2)|_L \rightarrow 0, $$
where $\mathcal{I}_{X \cup \Gamma_1}(2)|_P=\mathcal{I}_{\Gamma_1 \cup \{p,q\}}(2)|_P$, $\mathcal{I}_{X \cup \Gamma_1}(2)|_Q= \mathcal{I}_{X}(2)|_Q$ and $\mathcal{I}_{X \cup \Gamma_1}(2)|_L=\mathcal{I}_{\{p,q\}}(2)|_L$. 
Assume that $\Gamma_1 \cup \{p,q\}$ don't lie on a plane conic on $P$, then we have $h^0(\mathcal{I}_{X \cup \Gamma_1}(2)|_P)=0$. Besides, we have $h^i(\mathcal{I}_{X}(2)|_Q)=h^i(\mathcal{I}_{E}(2))=h^i(\mathcal{O}_{\mathbb{F}_2}(e+f))$. Thus $h^0(\mathcal{I}_{X}(2)|_Q)=2$ and $h^1(\mathcal{I}_{X}(2)|_Q)=0$. From this, we know what $$h^0(\mathcal{I}_{X \cup \Gamma_1}(2))=h^0(\mathcal{I}_{X \cup \Gamma_1}(2)|_P)+h^0(\mathcal{I}_{X \cup \Gamma_1}(2)|_Q)-h^0(\mathcal{I}_{X \cup \Gamma_1}(2)|_L)=0+2-1=1$$ and $h^1(\mathcal{I}_{X \cup \Gamma_1}(2))=0$. Furthermore, from the exact sequence $$0 \rightarrow \mathcal{I}_{X \cup \Gamma_1}(2) \rightarrow \mathcal{I}_{\Gamma}(2) \rightarrow \mathcal{I}_{\Gamma_2/X}(2) \rightarrow 0,$$ we have $h^0(\mathcal{I}_\Gamma(2))=h^0(\mathcal{I}_{\Gamma_2/X}(2))+h^0(\mathcal{I}_{X \cup \Gamma_1}(2))=1+1=2$. This is to say, the 10 points $p_1, \cdots, p_{10}$ impose independent conditions on quadrics on $S_4$. Therefore, 10 points fail to impose independent conditions only if $\Gamma_1 \cup \{p,q\}$ lie on a plane conic on $P$, as drawn below:

\begingroup

% 0 = smooth conic on P.
% 1 = split conic on P.
\def\SingularQuadricCase#1{%
\begin{tikzpicture}[
  line cap=round,
  line join=round,
  every node/.style={scale=0.85},
  bluecurve/.style={blue,semithick},
  hiddenblue/.style={
    blue,semithick,
    dash pattern=on 2pt off 2pt
  }
]

\path[use as bounding box]
  (-1.25,-0.35) rectangle (2.45,3.35);

% Plane vertices and cone vertex.
\coordinate (A) at (0.558,2.874);
\coordinate (B) at (1.442,0.126);
\coordinate (V) at (-0.70,2.00);
\coordinate (p) at (1.00,1.50);

% --------------------------------------------------
% Schematic cubic C through the cone vertex
% --------------------------------------------------

% The parameter t=0 gives p.
% Calculate the second intersection q with L.
\pgfmathsetmacro{\rulingSlope}{
  (2.874-1.50)/(1.00-0.558)
}

% Recalculate the second intersection of C with this ruling.
\pgfmathsetmacro{\qt}{
  (0.75-sqrt(0.75*0.75+0.12+1.40/\rulingSlope))/0.50
}

\pgfmathsetmacro{\qx}{
  1+0.12*\qt+0.75*\qt*\qt-0.25*\qt*\qt*\qt
}
\pgfmathsetmacro{\qy}{
  1.50+1.40*\qt
}

\coordinate (q) at (\qx,\qy);

% --------------------------------------------------
% Singular quadric Q and plane P
% --------------------------------------------------

% Both sides pass through the cone vertex p.
\draw (0,3)--(2,0);
\draw (0,0)--(2,3);

% Back arcs.
\draw
  (0,3)
  .. controls (0.5,3.2) and (1.5,3.2) .. (2,3);

\draw
  (0,0)
  .. controls (0.5,-0.2) and (1.5,-0.2) .. (2,0);

% Plane edges.
% The common ruling L is A--B.
\draw[white,line width=1.4pt]
  (A)--(V)--(B);

\draw[thin]
  (A)--(V)--(B)--cycle;

% Front arcs.
\draw
  (0,3)
  .. controls (0.5,2.8) and (1.5,2.8) .. (2,3);

\draw
  (0,0)
  .. controls (0.5,0.2) and (1.5,0.2) .. (2,0);

% --------------------------------------------------
% Blue cubic C on Q
% --------------------------------------------------

\draw[bluecurve]
  plot[
    domain=-0.96:0.95,
    samples=100,
    variable=\t
  ]
  ({1+0.12*\t+0.75*\t*\t-0.25*\t*\t*\t},
   {1.50+1.40*\t});

% Six red points on C, separate from p and q.
\foreach \t in {-0.86,-0.48,-0.25,0.25,0.55,0.85}{
  \fill[red]
    ({1+0.12*\t+0.75*\t*\t-0.25*\t*\t*\t},
     {1.50+1.40*\t})
    circle (1pt);
}

% --------------------------------------------------
% Blue plane conic through p and q
% --------------------------------------------------

\ifnum#1=0\relax

  % Smooth ellipse through p and q.
  \pgfmathsetmacro{\planeCx}{
    (1+\qx)/2-0.95*cos(65)
  }
  \pgfmathsetmacro{\planeCy}{
    (1.50+\qy)/2+0.65*cos(65)
  }
  \pgfmathsetmacro{\planeBx}{
    (1-\qx)/(2*sin(65))
  }
  \pgfmathsetmacro{\planeBy}{
    (1.50-\qy)/(2*sin(65))
  }

  % Full ellipse, initially dashed.
  \begin{scope}[
    cm={0.95,-0.65,\planeBx,\planeBy,
        (\planeCx,\planeCy)}
  ]
    \draw[hiddenblue]
      (0,0) circle[radius=1];
  \end{scope}

  % Solid portion inside the drawn plane.
  \begin{scope}
    \clip (A)--(V)--(B)--cycle;

    \begin{scope}[
      cm={0.95,-0.65,\planeBx,\planeBy,
          (\planeCx,\planeCy)}
    ]
      \draw[bluecurve]
        (0,0) circle[radius=1];
    \end{scope}
  \end{scope}

  % Four red points on the smooth plane conic.
  \foreach \ang in {110,155,205,245}{
    \fill[red]
      ({\planeCx+0.95*cos(\ang)+\planeBx*sin(\ang)},
       {\planeCy-0.65*cos(\ang)+\planeBy*sin(\ang)})
      circle (1pt);
  }

\else

  % Split conic: one line through p, another through q.
  \coordinate (T) at (-0.10,1.95);

  \coordinate (Pa) at ($(p)!-0.12!(T)$);
  \coordinate (Pb) at ($(p)!1.40!(T)$);
  \coordinate (Qa) at ($(q)!-0.12!(T)$);
  \coordinate (Qb) at ($(q)!1.28!(T)$);

  % Dashed extensions.
  \draw[hiddenblue]
    (Pa)--(Pb)
    (Qa)--(Qb);

  % Solid portions inside the drawn plane.
  \begin{scope}
    \clip (A)--(V)--(B)--cycle;

    \draw[bluecurve]
      (Pa)--(Pb)
      (Qa)--(Qb);
  \end{scope}

  % Two red points on each line.
  \foreach \t in {0.46,0.76}{
    \fill[red] ($(p)!\t!(T)$) circle (1pt);
    \fill[red] ($(q)!\t!(T)$) circle (1pt);
  }

\fi

% --------------------------------------------------
% Black intersection points and labels
% --------------------------------------------------

\fill[black] (p) circle (1.3pt);
\fill[black] (q) circle (1.3pt);

\node[right,inner sep=3pt,scale=0.7] at (p) {$p$};
\node[right,inner sep=3pt,scale=0.7] at (q) {$q$};

\node[scale=0.5] at (-0.85,1.98) {$P$};
\node[scale=0.5] at (1.95,2.65) {$Q$};
\node[scale=0.5] at (0.75,2.64) {$L$};

\node[blue,scale=0.5] at (1.25,2.48) {$X$};

\node[scale=0.5,red] at (0.45,2.04) {$\Gamma_1$};
\node[scale=0.5,red] at (0.98,2.00) {$\Gamma_2$};

\end{tikzpicture}%
}

% --------------------------------------------------
% Two pictures in one row
% --------------------------------------------------

\begin{center}

\begin{minipage}[t]{0.48\linewidth}
  \centering
  \resizebox{\linewidth}{!}{\SingularQuadricCase{0}}
  \par\smallskip
  \noindent
  \makebox[\linewidth][l]{%
    \hspace*{0.608\linewidth}%
    \makebox[0pt][c]{%
      \small Configuration (4) in Table \ref{tab:configurations}%
    }%
  }\par
\end{minipage}\hfill
\begin{minipage}[t]{0.48\linewidth}
  \centering
  \resizebox{\linewidth}{!}{\SingularQuadricCase{1}}
  \par\smallskip
  \noindent
  \makebox[\linewidth][l]{%
    \hspace*{0.608\linewidth}%
    \makebox[0pt][c]{%
      \small Configuration (6) in Table \ref{tab:configurations}%
    }%
  }\par
\end{minipage}

\end{center}
\endgroup

\textit{Case 2}: $p = q$. We have the following exact sequence $$0 \rightarrow \mathcal{I}_{X \cup \Gamma_1}(2) \rightarrow \mathcal{I}_{X \cup \Gamma_1}(2)|_P \oplus \mathcal{I}_{X \cup \Gamma_1}(2)|_Q \rightarrow \mathcal{I}_{X \cup \Gamma_1}(2)|_L \rightarrow 0, $$
where $\mathcal{I}_{X \cup \Gamma_1}(2)|_P=\mathcal{I}_{\Gamma_1 \cup \{2p\}}(2)|_P$, $\mathcal{I}_{X \cup \Gamma_1}(2)|_Q= \mathcal{I}_{X}(2)|_Q$ and $\mathcal{I}_{X \cup \Gamma_1}(2)|_L=\mathcal{I}_{\{2p\}}(2)|_L$. Similarly, 10 points fail to impose independent conditions only if the conic passing through $\Gamma_1 \cup p$ is tangent to $L$ at $p$, as drawn below:

\begingroup

% 0 = smooth conic on P.
% 1 = split conic on P.
\def\SingularCoincidentCase#1{%
\begin{tikzpicture}[
  line cap=round,
  line join=round,
  every node/.style={scale=0.85},
  bluecurve/.style={blue,semithick},
  hiddenblue/.style={
    blue,semithick,
    dash pattern=on 2pt off 2pt
  }
]

\path[use as bounding box]
  (-1.25,-0.35) rectangle (2.45,3.35);

% Plane vertices and cone vertex.
\coordinate (p) at (1.00,1.50);

% A lies on the upper rim.
\coordinate (A) at (0.5070987074,2.8801236193);

% B lies on the lower rim; A, p, B are collinear.
\coordinate (B) at ($(p)!-1!(A)$);

\coordinate (V) at (-0.70,2.00);

% --------------------------------------------------
% Singular quadric Q
% --------------------------------------------------

\draw (0,3)--(p)--(2,0);
\draw (0,0)--(p)--(2,3); 
% Back arcs.
\draw
  (0,3)
  .. controls (0.5,3.2) and (1.5,3.2) .. (2,3);

\draw
  (0,0)
  .. controls (0.5,-0.2) and (1.5,-0.2) .. (2,0);

% Plane P and the common ruling L=A--B.
\draw[white,line width=1.4pt]
  (A)--(V)--(B);

% Outside edges of the plane.
\draw[thin] (A)--(V)--(B);

% Common intersection line, explicitly through p=q.
\draw[thin] (A)--(p)--(B);

% Front arcs.
\draw
  (0,3)
  .. controls (0.5,2.8) and (1.5,2.8) .. (2,3);

\draw
  (0,0)
  .. controls (0.5,0.2) and (1.5,0.2) .. (2,0);

% --------------------------------------------------
% Schematic irreducible cubic X
% --------------------------------------------------

% At t=0, X passes through p with tangent parallel to L.
\draw[bluecurve]
  plot[
    domain=-1.15:1.15,
    samples=121,
    variable=\t
  ]
  ({1-0.50*\t+0.20*\t*\t+0.08*\t*\t*\t},
   {1.50+1.40*\t-0.224*\t*\t*\t});

% Six red points on X, distinct from p.
\foreach \t in {-1.10,-0.85,-0.60,0.60,0.85,1.10}{
  \fill[red]
    ({1-0.50*\t+0.20*\t*\t+0.08*\t*\t*\t},
     {1.50+1.40*\t-0.224*\t*\t*\t})
    circle (1pt);
}

% --------------------------------------------------
% Plane conic meeting L in 2p
% --------------------------------------------------

\ifnum#1=0\relax

  % Smooth conic tangent to L at p.
  \begin{scope}[
    cm={0.60,-0.23,-0.18,0.504,(0.40,1.73)}
  ]
    \draw[bluecurve]
      (0,0) circle[radius=1];
  \end{scope}

  % Four red points on the smooth plane conic.
  \foreach \ang in {70,130,190,250}{
    \fill[red]
      ({0.40+0.60*cos(\ang)-0.18*sin(\ang)},
       {1.73-0.23*cos(\ang)+0.504*sin(\ang)})
      circle (1pt);
  }

\else

  % Two distinct lines, both passing through p.
  \coordinate (Tone) at (-0.45,2.10);
  \coordinate (Ttwo) at (0.10,1.40);

  \coordinate (Ua) at ($(p)!-0.12!(Tone)$);
  \coordinate (Ub) at ($(p)!1.12!(Tone)$);
  \coordinate (Wa) at ($(p)!-0.12!(Ttwo)$);
  \coordinate (Wb) at ($(p)!1.25!(Ttwo)$);

  % Dashed extensions.
  \draw[hiddenblue]
    (Ua)--(Ub)
    (Wa)--(Wb);

  % Solid portions inside the drawn plane.
  \begin{scope}
    \clip (A)--(V)--(B)--cycle;

    \draw[bluecurve]
      (Ua)--(Ub)
      (Wa)--(Wb);
  \end{scope}

  % Two red points on each line.
  \foreach \t in {0.35,0.72}{
    \fill[red] ($(p)!\t!(Tone)$) circle (1pt);
    \fill[red] ($(p)!\t!(Ttwo)$) circle (1pt);
  }

\fi

% --------------------------------------------------
% Black vertex and labels
% --------------------------------------------------

\fill[black] (p) circle (1.3pt);

\node[right,inner sep=3pt,scale=0.6] at (p) {$p$};

\node[scale=0.5] at (-0.85,1.98) {$P$};
\node[scale=0.5] at (2.00,2.65) {$Q$};
\node[scale=0.5] at (0.48,2.64) {$L$};

\node[blue,scale=0.5] at (0.98,2.25) {$X$};

\node[scale=0.5,red] at (-0.05,2.10) {$\Gamma_1$};
\node[scale=0.5,red] at (1.03,2.60) {$\Gamma_2$};

\end{tikzpicture}%
}

% --------------------------------------------------
% Two pictures in one row
% --------------------------------------------------

\begin{center}

\begin{minipage}[t]{0.48\linewidth}
  \centering
  \resizebox{\linewidth}{!}{\SingularCoincidentCase{0}}
  \par\smallskip
  \noindent
  \makebox[\linewidth][l]{%
    \hspace*{0.608\linewidth}%
    \makebox[0pt][c]{%
      \small Configuration (11) in Table \ref{tab:configurations}
    }%
  }\par
\end{minipage}\hfill
\begin{minipage}[t]{0.48\linewidth}
  \centering
  \resizebox{\linewidth}{!}{\SingularCoincidentCase{1}}
  \par\smallskip
  \noindent
  \makebox[\linewidth][l]{%
    \hspace*{0.608\linewidth}%
    \makebox[0pt][c]{%
      \small Configuration (12) in Table \ref{tab:configurations}
    }%
  }\par
\end{minipage}

\end{center}
\endgroup

$\left(\mathrm{II}\right)$ Denote the image of $E$ under $b$ by $X$. $X$ is the union of a conic $A$ and a line $B$ in $Q$. $(e+2f)e=0$ shows that the conic doesn't meet the cone point $p$. 
%Denote the intersection of $A$ and $L$ by $q$, we have $p \neq q$. We consider the exact sequence $$0 \rightarrow \mathcal{O}_{X}(2) \rightarrow \mathcal{O}_A(2) \oplus \mathcal{O}_B(2) \cong \mathcal{O}_{\bb^1}(4) \oplus \mathcal{O}_{\bb^1}(2)  \rightarrow \mathcal{O}_Q(2)|_{A \cap B} \rightarrow 0.$$ From this, we get $h^0(\mathcal{O}_X(2))=5+3-1=7$. Surjectivity of the restriction map shows that $h^1(\mathcal{O}_X(2))=0$.
Similarly, there is the exact sequence $$0 \rightarrow \mathcal{I}_{X \cup \Gamma_1}(2) \rightarrow \mathcal{I}_{X \cup \Gamma_1}(2)|_P \oplus \mathcal{I}_{X \cup \Gamma_1}(2)|_Q \rightarrow \mathcal{I}_{X \cup \Gamma_1}(2)|_L \rightarrow 0, $$
where $\mathcal{I}_{X \cup \Gamma_1}(2)|_P=\mathcal{I}_{\Gamma_1 \cup \{p,q\}}(2)|_P$, $\mathcal{I}_{X \cup \Gamma_1}(2)|_Q= \mathcal{I}_{X}(2)|_Q$ and $\mathcal{I}_{X \cup \Gamma_1}(2)|_L=\mathcal{I}_{\{p,q\}}(2)|_L$. Moreover, we have $$0 \rightarrow \mathcal{I}_{X \cup \Gamma_1}(2) \rightarrow \mathcal{I}_{\Gamma}(2) \rightarrow \mathcal{I}_{\Gamma_2}(2)|_X \rightarrow 0.$$ Same as what we do in $\left(1\right)$, the 10 points fail to impose independent conditions only if $\Gamma_1 \cup \{p,q\}$ lie on a plane conic on $P$. This gives configurations (5) and (7) in Table \ref{tab:configurations}.

$\left(\mathrm{III}\right)$ The image of $E$ under $b$ in $Q$ is the union of three distinct lines $L_1, L_2$ and $L_3$. In this case, we have exact sequences $$0 \rightarrow \mathcal{I}_{X \cup \Gamma_1}(2) \rightarrow \mathcal{I}_{X \cup \Gamma_1}(2)|_P \oplus \mathcal{I}_{X \cup \Gamma_1}(2)|_Q \rightarrow \mathcal{I}_{X \cup \Gamma_1}(2)|_L \rightarrow 0,$$ where $\mathcal{I}_{X \cup \Gamma_1}(2)|_P=\mathcal{I}_{\Gamma_1 \cup \{2p\}}(2)|_P$, $\mathcal{I}_{X \cup \Gamma_1}(2)|_Q= \mathcal{I}_{X}(2)|_Q$ and $\mathcal{I}_{X \cup \Gamma_1}(2)|_L=\mathcal{I}_{\{2p\}}(2)|_L$. By using the same method as \textit{Case 2} in $\left(1\right)$, we conclude that 10 points fail to impose independent conditions only if there exists a conic passing through $\Gamma_1 \cup p$, which is tangent to $L$ at $p$. This gives configurations (14) and (15) in Table \ref{tab:configurations}.

\subsection{$S_3$ : union of a plane and a smooth quadric surface}
Similar argument shows that the only possibility for the 10 points to fail to impose independent condition is that 4 points lie on the plane and 6 points lie on the smooth quadric, all away from the attached ruling. Thus we have the following picture:
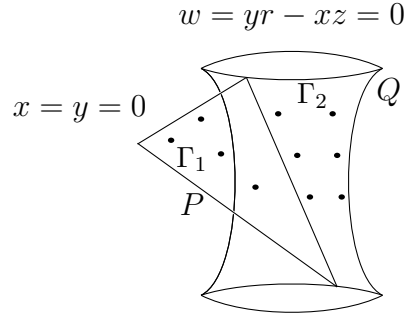
\[
\begin{tikzpicture}
\begin{scope}[xshift=4cm, yshift=0cm, xscale=1.2, yscale=1]

% Original quadric.
\draw
  (0,0)
  .. controls (0.5,0.3) and (0.5,2.7) .. (0,3);

\draw
  (0,0)
  .. controls (0.5,-0.3) and (1.5,-0.3) .. (2,0);

\draw
  (0,0)
  .. controls (0.5,0.3) and (0.5,2.7) .. (0,3);

\draw
  (2,0)
  .. controls (1.5,0.3) and (1.5,2.7) .. (2,3);

\draw
  (0,3)
  .. controls (0.5,3.3) and (1.5,3.3) .. (2,3);

% Original white underlays.
\draw[draw=white, line width=1pt]
  (0.5,2.88) -- (-0.7,2.01);

\draw[draw=white, line width=1pt]
  (-0.7,2.01) -- (1.5,0.11);

% Front arcs of the quadric.
\draw
  (0,3)
  .. controls (0.5,2.8) and (1.5,2.8) .. (2,3);

\draw
  (0,0)
  .. controls (0.5,0.2) and (1.5,0.2) .. (2,0);

% Original plane.
\draw[thin]
  (0.5,2.88) -- (1.5,0.11);

\draw[thin]
  (-0.7,2) -- (0.5,2.88);

\draw[thin]
  (-0.7,2.01) -- (1.5,0.11);

% Four points on P: Gamma_1.
\foreach \pt in {
  (-0.33,2.05),
  (0.00,2.33),
  (0.22,1.87),
  (0.60,1.43)
}{
  \fill \pt circle (1pt);
}

% Six points on Q: Gamma_2.
\foreach \pt in {
  (0.85,2.40),
  (1.45,2.40),
  (1.06,1.85),
  (1.50,1.85),
  (1.20,1.30),
  (1.55,1.30)
}{
  \fill \pt circle (1pt);
}

% Labels of the plane and quadric.
\node at (-0.10, 1.20) {$P$};
\node at (2.07,2.68) {$Q$};

% Labels of the two sets of points.
\node[font=\small] at (-0.10,1.80) {$\Gamma_1$};
\node[font=\small] at (1.23,2.64) {$\Gamma_2$};

% Equations.
\node[above] at (1.00,3.36)
  {$w=yr-xz=0$};

\node[above left] at (-0.48,2.14)
  {$x=y=0$};

% Configuration label.
%\node at (1,-1) {$S_3$};

\end{scope}
\end{tikzpicture}
\]
Therefore, $S_3$ is given by equations $wx=wy=yr-xz=0$. Again we add 2 general points on the plane $P$, and look for the necessary condition for these 12 points failing to impose independent conditions. Any quadric vanishing along the union of $\Gamma_1$ and the 2 points we added must vanish along the plane $P:\{x=y=0\}$. The quadric space vanishing along the plane $P$ is generated by $x^2,xy,xz,xr,y^2,yz,wx,wy$. Now we look at how many conditions do $\Gamma_2$ impose on this space. Note that every point of on $\Gamma_2$ satisfies $w=0$, so they impose no conditions on $xw$ and $yw$. Now we consider the space $V$ generated by $x^2,xy,xz,xr,y^2,yz$.

\begin{prop}
If $\Gamma_2$ do not lie on any $(1,2)$ class on $Q \simeq \mathbb{P}_{[s:t]}^1 \times \mathbb{P}_{[u:v]}^1$, then these 12 points impose independent conditions. In particular, our original 10 points impose independent conditions.
\end{prop}
Take $x=su, y=sv,r=tu, z=tv$. In terms of the coordinates of $\mathbb{P}^1$, $V=s\cdot H^0(\mathcal{O}_{\mathbb{P}^1 \times \mathbb{P}^1}(1,2))$. Since $\Gamma_2$ lie away from $x=y=0$, we have $s \neq 0$. Consider the evaluation map $$ev:V=s\cdot H^0(\mathcal{O}_{\mathbb{P}^1 \times \mathbb{P}^1}(1,2)) \rightarrow H^0(\mathcal{O}_{\Gamma_2}(2)).$$
Since both spaces have dimension 6, $ev$ fails to be surjective if and only if $ev$ has a kernel. Since $\Gamma_2$ lie away from $x=y=0$, we have $s \neq 0$. That is, $\Gamma_2$ lie on a bidegree $(1,2)$ curve $C'$ if and only if $\Gamma_2$ fail to impose independent conditions on $V$.  \qed

Now we assume that $\Gamma_2$ all lie on a bidegree $(1,2)$ curve $C' \subset Q$.

Note that $C'$ might be reducible, and we have the following possibilities:

$\left(\mathrm{I}\right)$ $C'$ is irreducible.

$\left(\mathrm{II}\right)$ $C'$ is the union of a conic $C$ and a line $L$, namely a union of a bidegree $(1,1)$ curve and a bidegree $(0,1)$ curve.

$\left(\mathrm{III}\right)$ $C'$ is the union of three rulings, namely a union of two distinct bidegree $(0,1)$ curves $L_1, L_2$ and a bidegree $(1,0)$ curve $L_3$.

\begin{remark}
    $C'$ is the union of a bidegree $(1,0)$ curve and a bidegree $(0,2)$ curve is impossible. Indeed, we can't have non-reduce structure in this case because we assume that any 4 of them can't be collinear.
\end{remark}

Denote the 4 points on $P$ by $\Gamma_1$, and the 6 points on $Q$ by $\Gamma_2$. Denote their union by $\Gamma$. Let $l:=\{x=y=w=0\}$.

We will give the computation for $\left(\mathrm{II}\right)$. And other cases will be similar.

\begin{prop}
    Let $q$ be the intersection of $L$ and $\ell$, $p$ be the intersection of $C$ and $\ell$ and $r$ be the intersection of $C$ and $L$. If $p \neq q$, then $\Gamma$ fail to impose independent conditions only if $\Gamma_1$ and $\{p,q\}$ lie on a plane conic on $P$. If $p=q$, then $\Gamma$ fail to impose independent conditions only if $\Gamma_1 \cup p$ lie on a plane conic on $P$, which is tangent to $\ell$ at p.
\end{prop}

\textit{Proof}. We first assume $p \neq q$, and $\Gamma_1 \cup \{p,q\}$ don't lie on a plane conic on $P$. We prove $\Gamma$ impose independent conditions on quadrics on $S_3$.
Consider the following sequence $$0 \rightarrow \mathcal{I}_{C' \cup \Gamma_1}(2) \rightarrow \mathcal{I}_{C' \cup \Gamma_1}(2)|_P \oplus \mathcal{I}_{C' \cup \Gamma_1}(2)|_Q \rightarrow \mathcal{I}_{\{p,q\}}(2)|_l,$$
where we have $\mathcal{I}_{C' \cup \Gamma_1}(2)|_P=\mathcal{I}_{\Gamma_1 \cup \{p,q\}}(2)|_P$ and $\mathcal{I}_{C' \cup \Gamma_1}(2)|_Q=\mathcal{I}_C'(2)|_Q$.

The assumption that $\Gamma_1 \cup \{p,q\}$ don't lie on a plane conic on $P$ implies $H^0(\mathcal{I}_{\Gamma_1 \cup \{p,q\}}(2)|_P)=H^1(\mathcal{I}_{\Gamma_1 \cup \{p,q\}}(2)|_P)=0$. Furthermore, we have $$\mathcal{I}_C'(2)|_Q = \mathcal{O}_{\mathbb{P}^1 \times \mathbb{P}^1}(-1,-2) \otimes \mathcal{O}_{\mathbb{P}^1 \times \mathbb{P}^1}(2,2)= \mathcal{O}_{\mathbb{P}^1 \times \mathbb{P}^1}(1,0).$$ Therefore, $h^0(\mathcal{I}_C'(2)|_Q)=2$ and $h^1(\mathcal{I}_C'(2)|_Q)=0$. Since $\mathcal{I}_C'(2)|_Q \rightarrow \mathcal{I}_{\{p,q\}}(2)|_l$ is surjective, we have $h^0(\mathcal{I}_{C' \cup \Gamma_1}(2))=1$ and $h^1(\mathcal{I}_{C' \cup \Gamma_1}(2))=0$.  Besides, we have $$0 \rightarrow \mathcal{O}_C'(2)(-\Gamma_2) \rightarrow \mathcal{O}_C'(2)(-\Gamma_2)|_C \oplus \mathcal{O}_C'(2)(-\Gamma_2)|_L \rightarrow \mathcal{O}_C'(2)(-\Gamma_2)|_r \rightarrow 0,$$ where $\mathcal{O}_C'(2)(-\Gamma_2)|_L = \mathcal{O}_{\mathbb{P}^1}(2-2)=\mathcal{O}_{\mathbb{P}^1}$, $\mathcal{O}_C'(2)(-\Gamma_2)|_C=\mathcal{O}_{\mathbb{P}^1}$ and $\mathcal{O}_C'(2)(-\Gamma_2)|_r=\mathcal{O}_r$. Thus we get \begin{align*}
h^0(\mathcal{O}_C'(2)(-\Gamma_2)) &= h^0(\mathcal{I}_{\Gamma_2}(2)|_C') \\
&= h^0(\mathcal{O}_{\mathbb{P}^1}) + h^0(\mathcal{O}_{\mathbb{P}^1}) - h^0(\mathcal{O}_{r}) \\
&= 1.
\end{align*}
%$h^0(\mathcal{O}_X(2)(-\Gamma_2))=h^0(\mathcal{I}_{\Gamma_2}(2)|_X)=1+1-1=1.$ 

Finally, consider the sequence $$0 \rightarrow \mathcal{I}_{C' \cup \Gamma_1}(2) \rightarrow \mathcal{I}_{\Gamma}(2) \rightarrow \mathcal{I}_{\Gamma_2}(2)|_C' \rightarrow 0.$$ We know that $h^0(\mathcal{I}_{\Gamma}(2))=h^0(\mathcal{I}_{C' \cup \Gamma_1}(2))+h^0(\mathcal{I}_{\Gamma_2}(2)|_C')=2$. In other words, $\Gamma$ impose independent conditions on quadrics on $S_3$.\\

Now we assume that $p=q$, and $\Gamma_1 \cup {p} $ don't lie on a plane conic on $P$ which is tangent to $\ell$ at $p$. Similarly, consider the following sequence $$0 \rightarrow \mathcal{I}_{C' \cup \Gamma_1}(2) \rightarrow \mathcal{I}_{C' \cup \Gamma_1}(2)|_P \oplus \mathcal{I}_{C' \cup \Gamma_1}(2)|_Q \rightarrow \mathcal{I}_{\{2p\}}(2)|_\ell,$$
where we have $\mathcal{I}_{C' \cup \Gamma_1}(2)|_P=\mathcal{I}_{\Gamma_1 \cup \{2p\}}(2)|_P$ and $\mathcal{I}_{C' \cup \Gamma_1}(2)|_Q=\mathcal{I}_C'(2)|_Q$. The assumption that $\Gamma_1 \cup {p} $ don't lie on a plane conic on $P$ which is tangent to $\ell$ at $p$ implies $h^0(\mathcal{I}_{\Gamma_1 \cup \{2p\}}(2)|_P)=h^1(\mathcal{I}_{\Gamma_1 \cup \{2p\}}(2)|_P)=0$. Therefore, we have $h^0(\mathcal{I}_{C' \cup \Gamma_1}(2))=1$ and $h^1(\mathcal{I}_{C' \cup \Gamma_1}(2))=0$. Same computation as above shows that $h^0(\mathcal{O}_C'(2)(-\Gamma_2))=h^0(\mathcal{I}_{\Gamma_2}(2)|_C')=1$. From the short exact sequence $$0 \rightarrow \mathcal{I}_{C' \cup \Gamma_1}(2) \rightarrow \mathcal{I}_{\Gamma}(2) \rightarrow \mathcal{I}_{\Gamma_2}(2)|_C' \rightarrow 0,$$ we know that $h^0(\mathcal{I}_{\Gamma}(2))=2$. Thus $\Gamma$ impose independent conditions. \qed

Therefore, 10 points fail to impose independent conditions only if one of the following situations occurs:
\begingroup

% 0 = smooth conic on P.
% 1 = split conic on P.
\def\QuadricConicCase#1{%
\begin{tikzpicture}[
  line cap=round,
  line join=round,
  bluecurve/.style={blue,semithick},
  hiddenblue/.style={
    blue,semithick,
    dash pattern=on 2pt off 2pt
  }
]

% Common bounding box.
\path[use as bounding box]
  (-1.25,-0.35) rectangle (2.45,3.35);

% Plane vertices and q.
\coordinate (A) at (0.50,2.88);
\coordinate (B) at (1.50,0.11);
\coordinate (V) at (-0.70,2.01);
\coordinate (q) at (0.93,1.6889);
\coordinate (Ltop) at (1.30,2.86);

% --------------------------------------------------
% Coordinates of p and r
% --------------------------------------------------

% C has parametrization
% (1+0.65 cos(t), 0.9+0.2 cos(t)+0.22 sin(t)).
% Choose p on both C and ell.
% Adjusted center and horizontal radius of C.
\pgfmathsetmacro{\Ccx}{0.9759121284}
\pgfmathsetmacro{\Cradius}{0.6638223632}

% Intersection p of C with ell.
\pgfmathsetmacro{\pCoeff}{
  0.2+2.77*\Cradius
}
\pgfmathsetmacro{\pConst}{
  4.265-2.77*\Ccx-0.9
}
\pgfmathsetmacro{\pAngle}{
  atan2(0.22,\pCoeff)
  -acos(\pConst/sqrt(\pCoeff*\pCoeff+0.22*0.22))
}
\pgfmathsetmacro{\px}{
  \Ccx+\Cradius*cos(\pAngle)
}
\pgfmathsetmacro{\py}{
  0.9+0.2*cos(\pAngle)+0.22*sin(\pAngle)
}
\coordinate (p) at (\px,\py);

% L passes through q.
% Its intersection with the back arc of C is r.
\pgfmathsetmacro{\Lslope}{
  (2.86-1.6889)/(1.30-0.93)
}

\pgfmathsetmacro{\rCoeff}{
  0.2-\Cradius*\Lslope
}
\pgfmathsetmacro{\rAngle}{
  atan2(0.22,\rCoeff)
  -acos(
    ((\Ccx-0.93)*\Lslope+0.7889)
    /sqrt(\rCoeff*\rCoeff+0.22*0.22)
  )
}
\coordinate (r) at
  ({\Ccx+\Cradius*cos(\rAngle)},
   {0.9+0.2*cos(\rAngle)+0.22*sin(\rAngle)});

\coordinate (Lbottom) at
  ({0.93+(0.11-1.6889)/\Lslope},0.11);

% --------------------------------------------------
% Original quadric and plane
% --------------------------------------------------

\draw
  (0,0)
  .. controls (0.5,0.3) and (0.5,2.7) .. (0,3);

\draw
  (2,0)
  .. controls (1.5,0.3) and (1.5,2.7) .. (2,3);

\draw
  (0,3)
  .. controls (0.5,3.3) and (1.5,3.3) .. (2,3);

\draw
  (0,0)
  .. controls (0.5,-0.3) and (1.5,-0.3) .. (2,0);

% Plane edges with white underlays.
\draw[white,line width=1.4pt]
  (A)--(V)--(B);

\draw[thin]
  (A)--(V)--(B)--cycle;

% Front arcs of the quadric.
\draw
  (0,3)
  .. controls (0.5,2.8) and (1.5,2.8) .. (2,3);

\draw
  (0,0)
  .. controls (0.5,0.2) and (1.5,0.2) .. (2,0);

% --------------------------------------------------
% Blue L and C on Q
% --------------------------------------------------

% L is solid above q and dashed below q.
\draw[bluecurve] (Ltop)--(q);
\draw[hiddenblue] (q)--(Lbottom);

% C meets the two outlines at the endpoints of these arcs.
\begin{scope}[
  cm={\Cradius,0.2,0,0.22,(\Ccx,0.9)}
]
  % Back arc.
  \draw[hiddenblue]
    (1.395855:1)
    arc[
      start angle=1.395855,
      end angle=176.553365,
      radius=1
    ];

  % Front arc.
  \draw[bluecurve]
    (176.553365:1)
    arc[
      start angle=176.553365,
      end angle=361.395855,
      radius=1
    ];
\end{scope}
% --------------------------------------------------
% Blue conic on P through Gamma_1, p, and q
% --------------------------------------------------

\ifnum#1=0\relax

  % Smooth ellipse through p and q.
  \pgfmathsetmacro{\planeCx}{
    (0.93+\px)/2-0.85*cos(65)
  }
  \pgfmathsetmacro{\planeCy}{
    (1.6889+\py)/2+0.50*cos(65)
  }
  \pgfmathsetmacro{\planeBx}{
    (0.93-\px)/(2*sin(65))
  }
  \pgfmathsetmacro{\planeBy}{
    (1.6889-\py)/(2*sin(65))
  }

  % Full ellipse, initially dashed.
  \begin{scope}[
    cm={0.85,-0.50,\planeBx,\planeBy,
        (\planeCx,\planeCy)}
  ]
    \draw[hiddenblue] (0,0) circle[radius=1];
  \end{scope}

  % Solid portion inside the drawn plane.
  \begin{scope}
    \clip (A)--(V)--(B)--cycle;

    \begin{scope}[
      cm={0.85,-0.50,\planeBx,\planeBy,
          (\planeCx,\planeCy)}
    ]
      \draw[bluecurve] (0,0) circle[radius=1];
    \end{scope}
  \end{scope}

  % Four marked points on this conic.
  \foreach \ang in {115,155,200,240}{
    \fill[red]
      ({\planeCx+0.85*cos(\ang)+\planeBx*sin(\ang)},
       {\planeCy-0.50*cos(\ang)+\planeBy*sin(\ang)})
      circle (1pt);
  }

\else

  % Split conic: two lines meeting at T.
  % One passes through q, the other through p.
  \coordinate (T) at (-0.03,1.91);

  \coordinate (Qa) at ($(q)!-0.12!(T)$);
  \coordinate (Qb) at ($(q)!1.70!(T)$);
  \coordinate (Pa) at ($(p)!-0.12!(T)$);
  \coordinate (Pb) at ($(p)!1.35!(T)$);

  % Dashed extensions.
  \draw[hiddenblue]
    (Qa)--(Qb)
    (Pa)--(Pb);

  % Solid portions inside the drawn plane.
  \begin{scope}
    \clip (A)--(V)--(B)--cycle;
    \draw[bluecurve]
      (Qa)--(Qb)
      (Pa)--(Pb);
  \end{scope}

  % Two marked points on each line.
  \foreach \t in {0.45,0.80}{
    \fill[red] ($(q)!\t!(T)$) circle (1pt);
    \fill[red] ($(p)!\t!(T)$) circle (1pt);
  }

\fi

% --------------------------------------------------
% Intersection points and labels
% --------------------------------------------------

\fill (q) circle (1.2pt);
\fill (p) circle (1.2pt);
\fill (r) circle (1pt);

\node[right,inner sep=3pt,scale=0.5] at (q) {$q$};
\node[below right,inner sep=2pt, scale=0.5] at (p) {$p$};
\node[font=\small,scale=0.5] at ($(r)+(0.05,-0.13)$) {$r$};

\node [scale=0.5] at (-0.83,1.98) {$P$};
\node  [scale=0.5] at (2.00,2.65) {$Q$};

%\node[font=\small] at (-0.16,2.45) {$\Gamma_1$};
\node [scale=0.5] at (0.50,2.55) {$\ell$};

\node[blue,scale=0.5] at (1.30,2.42) {$L$};
\node[blue,scale=0.5] at (1.84,1.03) {$C$};

% Equations retained from the previous picture.
%\node[above] at (1.00,3.36){$w=yr-xz=0$};

%\node[above left] at (-0.48,2.94){$x=y=0$};
% Two red points on L.
\foreach \t in {0.30,0.65}{
  \fill[red]
    ($(Ltop)!\t!(q)$)
    circle (1pt);
}

% Four red points on C.
\foreach \ang in {215,255,315,340}{
  \fill[red]
    ({\Ccx+\Cradius*cos(\ang)},
     {0.9+0.2*cos(\ang)+0.22*sin(\ang)})
    circle (1pt);
}
\end{tikzpicture}%
}

% --------------------------------------------------
% Two pictures in one row
% --------------------------------------------------

\begin{center}

\begin{minipage}[t]{0.48\linewidth}
  \centering
  \resizebox{\linewidth}{!}{\QuadricConicCase{0}}
  \par\smallskip
  \noindent
  \makebox[\linewidth][l]{%
    \hspace*{0.608\linewidth}%
    \makebox[0pt][c]{%
      \small Configuration (5) in Table \ref{tab:configurations}%
    }%
  }\par
\end{minipage}\hfill
\begin{minipage}[t]{0.48\linewidth}
  \centering
  \resizebox{\linewidth}{!}{\QuadricConicCase{1}}
  \par\smallskip
  \noindent
  \makebox[\linewidth][l]{%
    \hspace*{0.608\linewidth}%
    \makebox[0pt][c]{%
      \small Configuration (7) in Table \ref{tab:configurations}%
    }%
  }\par
\end{minipage}

\end{center}
\endgroup

For case $\left(\mathrm{I}\right)$, we denote the intersection of $C'$ with $P$ to be $p$ and $q$ respectively ($p$ and $q$ might collide). Similar computation shows that for $p \neq q$,  $\Gamma$ fails to impose independent conditions only if $\Gamma_1$ and $\{p,q\}$ lie on a plane conic on $P$, i.e., configurations (4) and (6) in Table \ref{tab:configurations}; if $p=q$, then $\Gamma$ fail to impose independent conditions only if $\Gamma_1 \cup p$ lie on a plane conic on $P$, which is tangent to $\ell$ at p, i.e., configurations (11) and (12) in Table \ref{tab:configurations}. For case $\left(\mathrm{III}\right)$, we denote $L_1 \cap l = p$ and $L_2 \cap l = q$. Since $L_1 \cap L_2 = \emptyset$, $p$ and $q$ cannot collide. Similar computation shows that $\Gamma$ fails to impose independent conditions only if $\Gamma_1$ and $\{p,q\}$ lie on a plane conic on $P$, i.e., configurations (7) and (8) in Table \ref{tab:configurations}.

\subsection{$S_2$ : singular cubic scroll} \label{3.5}

By blowing up $S_2$ via its cone point $p$, we have the following diagram: 
\[\begin{tikzcd}
    \mathbb{F}_3 \arrow[r, hookrightarrow, "|e+3f|"] \arrow[d,"b"] &\Bl_p\bb^4 \arrow[d] \\
    S_2 \arrow[r] & \bb^4
\end{tikzcd}\] 

Take any curve $E$ of class $e+5f$ in $\mathbb{F}_3$, and we now compute its genus and degree. By adjunction, we have $$2p_a(E)-2=E(E+K_{\mathbb{F}_3})$$ Therefore, $p_a(E)=1/2((e+5f)(-e)+2)=0$. Since the hyperplane class is given by $e+3f$, we get $\deg E = (e+5f)(e+3f)=5$. Thus $E$ is a genus 0 curve with degree 5. Since $H^0(\mathcal{O}(e+5f)) \simeq H^0(\mathcal{O}_{\mathbb{P}^1}(2)) \oplus H^0(\mathcal{O}_{\mathbb{P}^1}(5))$, we have $h^0(\mathcal{O}(e+5f))=9$. That is, if we have $8$ distinct points in $\mathbb{F}_3$, we can find an element $E$ in the linear series $|e+5f|$, which passes through these 8 points.

\begin{remark}
    We need to consider the case when the cone point is one of the 10 marked points. Note that $(2e+6f)e=0$ implies that for any quadric $Q$ on $\mathbb{F}_3$, it either contains the exceptional divisor $e$, or it does not intersect $e$ at all. That is, if a quadric vanishes at a point of $e$, then it must vanish along $e$. So if $p$ is one of the 10 marked points, we can view it as a single point on $e$.
\end{remark}
Denote the set of 10 marked points $p_1, \dots, p_{10}$ by $\Gamma$. Since any 9 points cannot be collinear, we may assume that $p_9$ and $p_{10}$ do not lie on the same fiber, and they both do not come from the cone point. Take $E \in |e+5f|$ such that it passes through $p_1, \dots, p_8$. Furthermore, if the preimage of $\Gamma$ impose independent conditions on quadrics on $\mathbb{F}_3$, then $\Gamma$ impose independent conditions on quadrics on $S_2$.

\begin{prop}
Assume $p_9$ and $p_{10}$ do not lie on $E$, then we have $\Gamma$ impose independent conditions on quadrics on $\mathbb{F}_3$.    
\end{prop}

\textit{Proof}. Denote the set of $p_1, \dots, p_8$ by $\Gamma_1$. We now consider short exact sequences on $\mathbb{F}_3$.

We have 
\[\begin{tikzcd}
0 \arrow[r] & \mathcal{I}_E(2) \arrow[r] & \mathcal{I}_{\Gamma_1}(2) \arrow[r] & \mathcal{I}_{\Gamma_1}(2)|_E \arrow[r] & 0 \\
0 \arrow[r] & \mathcal{I}_{E \cup \{p_9, p_{10}\}}(2) \arrow[u] \arrow[r] & \mathcal{I}_{\Gamma_1 \cup \{p_9, p_{10}\}}(2) \arrow[u] \arrow[r] & \mathcal{I}_{\Gamma_1}(2)|_E \arrow[u] \arrow[r] & 0
\end{tikzcd}\]
Note that $\mathcal{I}_E(2) = \mathcal{O}(-e-5f+2e+6f))=\mathcal{O}(e+f)$. Thus $h^0(\mathcal{I}_E(2))=2$ and $h^1(\mathcal{I}_E(2))=1$. Since points in $\Gamma_1$ impose independent conditions on quadrics on $\mathbb{F}_3$, we have $h^0(\mathcal{I}_{\Gamma_1}(2))=4$ and $h^1(\mathcal{I}_{\Gamma_1}(2))=0$. In particular, the connecting morphism $H^0(\mathcal{I}_{\Gamma_1}(2)|_E) \rightarrow H^1(\mathcal{I}_E(2))$ is surjective. We also have $\mathcal{I}_{E \cup \{p_9, p_{10}\}}(2)=\mathcal{I}_{\{p_9, p_{10}\}}(e+f)$. Now consider $$0 \rightarrow \mathcal{I}_{\{p_9, p_{10}\}}(e+f) \rightarrow \mathcal{O}_{\mathbb{F}_3}(e+f) \rightarrow \mathcal{O}_{\{p_9,p_{10}\}}(e+f) \rightarrow 0.$$ By our assumption that $p_9, p_{10}$ do not lie on the same fiber and away from $e$, we have the section of $\mathcal{O}_{\mathbb{F}_3}(e+f)$ cannot vanish along both $p_9$ and $p_{10}$. Since $h^0(\mathcal{O}_{\mathbb{F}_3}(e+f))=h^0(\mathcal{O}_{\{p_9,p_{10}\}}(e+f))=2$, we have $$H^0(\mathcal{O}_{\mathbb{F}_3}(e+f)) \xrightarrow{\sim} H^0(\mathcal{O}_{\{p_9,p_{10}\}}(e+f)).$$ This implies that $h^0(\mathcal{I}_{\{p_9, p_{10}\}}(e+f))=0$ and $H^1(\mathcal{I}_{\{p_9, p_{10}\}}(e+f)) \xrightarrow{\sim} H^1(\mathcal{O}_{\mathbb{F}_3}(e+f))$ is an isomorphism. Therefore, we have the following commutative diagram: 
\[\begin{tikzcd}
H^0(\mathcal{I}_{\Gamma_1}(2)|_E) \arrow[r, "\delta'"] \arrow[d, equal, "id"'] & H^1(\mathcal{I}_{E \cup \{p_9, p_{10}\}}(2)) \arrow[d, "\sim", "\beta"'] \\
H^0(\mathcal{I}_{\Gamma_1}(2)|_E) \arrow[r, "\delta"] & H^1(\mathcal{O}_{\mathbb{F}_3}(e+f))
\end{tikzcd}\]
Since $\delta$ is surjective, we have $\delta'$ is surjective too. Therefore, we get
\begin{align*}
h^0(\mathcal{I}_{\Gamma_1 \cup \{p_9, p_{10}\}}(2)) &= h^0(\mathcal{I}_{E \cup \{p_9, p_{10}\}}(2)) + h^0(\mathcal{I}_{\Gamma_1}(2)|_E) - 1 \\
&= 0 + 3 - 1 \\
&= 2.
\end{align*}
That is, $\Gamma=\Gamma_1 \cup \{p_9, p_{10}\}$ imposes independent conditions on $\mathbb{F}_3$. \qed

\begin{prop}
Assume $p_1, \dots, p_9 \in E$, and $p_{10} \notin E$, where $p_{10}$ does not come from the cone point. Then $\Gamma$ imposes independent conditions on quadrics on $\mathbb{F}_3$.
\end{prop}

\textit{Proof}. Denote $p_1, \dots, p_9$ by $\Gamma_2$. As in the computation above, from exact sequence $$0 \rightarrow  \mathcal{I}_E(2) \rightarrow \mathcal{I}_{\Gamma_2}(2) \rightarrow \mathcal{I}_{\Gamma_2}(2)|_E \rightarrow 0,$$ we have $h^0(\mathcal{I}_{\Gamma_2}(2)|_E)=2$ and $h^1(\mathcal{I}_{\Gamma_2}(2)|_E)=0$. Then consider $$0 \rightarrow \mathcal{I}_{E \cup \{p_{10}\}}(2) \rightarrow \mathcal{I}_E(2) \rightarrow \mathcal{O}_{p_{10}}(2) \rightarrow 0,$$ we have $h^0(\mathcal{I}_{E \cup \{p_{10}\}}(2))=h^1(\mathcal{I}_{E \cup \{p_{10}\}}(2))$. Moreover, we have $h^0(\mathcal{I}_{E \cup \{p_{10}\}}(2))=h^0(\mathcal{I}_{p_{10}}(e+f))=1$. We get $h^1(\mathcal{I}_{E \cup \{p_{10}\}}(2))=1$. Finally, we consider $$0 \rightarrow \mathcal{I}_{E \cup \{p_{10}\}}(2) \rightarrow \mathcal{I}_{\Gamma_2 \cup \{p_{10}\}}(2) \rightarrow \mathcal{I}_{\Gamma_2}(2)|_E \rightarrow 0.$$ Similar to the argument above, $H^0(\mathcal{I}_{\Gamma_2}(2)|_E) \rightarrow H^1(\mathcal{I}_{E \cup \{p_{10}\}}(2))$ is surjective. 

Thus we claim that
\begin{align*}
h^0(\mathcal{I}_{\Gamma_2 \cup \{p_{10}\}}(2)) &= h^0(\mathcal{I}_{E \cup \{p_{10}\}}(2)) + h^0(\mathcal{I}_{\Gamma_2}(2)|_E) - 1 \\
&= 1 + 2 - 1 \\
&= 2.
\end{align*}
That is, $\Gamma=\Gamma_2 \cup \{p_{10}\}$ impose independent conditions on quadrics on $\mathbb{F}_3$. \qed

Therefore, $\Gamma$ fail to impose independent conditions on quadrics only if they lie on the image of $E$ in $S_2$. 

Since $E \in |e+5f|$, we have $E \cdot e = (e+5f)e=2$. Therefore, $E$ either contains $e$ or intersection $e$ in a scheme of length two. If $E$ contains $e$, then its image will be a genus 1 curve of degree 5. For the second case, 2 points collapse when taking the image under the contraction. Thus the image of $E$ under the contraction will be a curve of degree 5 and genus 1. If $E$ intersects $e$ at 2 distinct points, the image is the specialization of the nodal curve of degree 5; if $E$ intersects $E$ at a point with multiplicity 2, then the image is the specialization of a cuspidal curve of degree 5. However, $E$ may not be irreducible. To consider its image under the blow up, we are going to consider all the cases of class $e+5f$ on $\mathbb{F}_3$. 
Following are all the possibilities, with the corresponding possible configurations in Table \ref{tab:configurations}:

$\left(\mathrm{I}\right)$ $E$ is irreducible: Configuration (2) and (9).

$\left(\mathrm{II}\right)$ $E$ is the union of irreducible classes $e+4f,f$: Configuration (3) and (10).

$\left(\mathrm{III}\right)$ $E$ is the union of irreducible classes $e+3f,f$ and $f$: Configuration (6).

$\left(\mathrm{IV}\right)$ $E$ is the union of irreducible classes $e$,$f$,$f$,$f$,$f$ and $f$: Configuration (15).

\begin{remark}
    It is not possible for $E$ to be the union of irreducible classes $e+2f$, $f$, $f$ and $f$. In fact, $(e+2f)e=-1$ forces $e+2f$ to contain $e$, and thus not irreducible. 
\end{remark}

Therefore, $\Gamma$ fails to impose independent conditions on quadrics only if it is one of (2), (3), (6), (9), (10), (15) in the configurations in Table \ref{tab:configurations}.

\subsection{$S_1$ : the smooth cubic scroll}

In this case, we have \begin{tikzcd}
    S_1 \simeq \Bl_p(\mathbb{P}^2) \arrow[r, hookrightarrow, "|e+2f|"] & \mathbb{P}^4.
\end{tikzcd} Take any curve of class $2e+3f$ in $S_1$, and we now compute its genus and degree. By adjunction, we have $$2p_a(E)-2=E(E+K_{S_1})$$ Therefore, $p_a(E)=1/2((2e+3f)(2e+3f-2e-3f)+2)=1$. Since the hyperplane class is given by $e+2f$, we get $\deg E = (2e+3f)(e+2f)=5$. Thus $E$ is a genus 1 curve of degree 5. Note that in this case, we again have $h^0(2e+3f)=9$. Use a similar argument as in section \ref{3.5}, $\Gamma$ fails to impose independent conditions only if $\Gamma \subset E$. 

Now we list all the possibilities of a curve $E$ with class $2e+3f$, with the corresponding possible configurations in Table \ref{tab:configurations}:

$\left(\mathrm{I}\right)$ $E$ is irreducible of class $2e+3f$: Configuration (1), (2) and (9).

$\left(\mathrm{II}\right)$ $E$ is the union of irreducible classes $2e+2f$ and $f$: Configuration (3) and (10).

$\left(\mathrm{III}\right)$ $E$ is the union of irreducible classes $e+2f$ and $e+f$: Configuration (4) and (11).

$\left(\mathrm{IV}\right)$ $E$ is the union of irreducible classes $e+3f$ and $e$: Configuration (3) and (10). 

$\left(\mathrm{V}\right)$ $E$ is the union of irreducible classes $e+2f,e$ and $f$: Configuration (6) and (12).

$\left(\mathrm{VI}\right)$ $E$ is the union of irreducible classes $e+f, e+f$ and $f$: Configuration (5) and (13).

$\left(\mathrm{VII}\right)$ $E$ is the union of irreducible classes $e+f,e,f$ and $f$: Configuration (7).

Therefore, $\Gamma$ fails to impose independent conditions on quadrics only if it is one of (1),(2), (3),(4),(5),(6),(7),(9),(10),(11),(12),(13) in the configurations of Table \ref{tab:configurations}.\\

As a summary of this section, we have: 
\begin{prop}
$\Gamma$ fails to impose independent conditions on quadrics only if it is one of the 15 configurations of Table \ref{tab:configurations}.
\label{prop1}
\end{prop}

\section{Proof of Theorem \ref{thm}}\label{Sectionproof} 

First, we assume $p_1, \dots, p_{10}$ fail to impose independent conditions. By proposition \ref{prop1}, $\Gamma$ lies on a curve with one of the 15 configurations. Note that for a fixed component of $E$, not all quadrics vanishing at $p_1, \dots, p_{10}$ will contain this component. Otherwise there exists three quadrics (which form a complete intersection) whose intersection contains a component of $E$. This contradicts the fact that $C$ is irreducible. Since containing a component is a closed condition and we have finite number of components for every configuration of $E$, we can find $Q$ vanishing along $p_1, \dots, p_{10}$, such that $Q$ does not contain any component of $E$. By B\'ezout's theorem, $|Q \cap E|=10$ with multiplicity. Furthermore, we have $\{p_1, \dots, p_{10}\} \subset Q \cap E$. Thus we get $\{p_1, \dots, p_{10}\} = Q \cap E$. In particular, $p_1, \dots, p_{10}$ are smooth points on $E$, and $p_1, \dots, p_{10}$ is a transverse intersection of a quadric and $E$. Thus $\mathcal{O}_E(p_1+ \cdots + p_{10}) \simeq \mathcal{O}_E(2)$. \\

Conversely, assume that there is $E$ of the configurations, such that $\mathcal{O}_E(p_1+ \cdots + p_{10}) \simeq \mathcal{O}_E(2)$. We first prove we can find a smooth curve as a complete intersection of three quadrics passing through these 10 points $p_1, \dots, p_{10}$. Note that we have the inclusion $$H^0(\mathcal{I}_E(2)) \subset H^0(\mathcal{I}_{\{p_1, \dots, p_{10}\}}(2)).$$ From appendix \ref{quadrics}, we know that the ideal sheaf of $E$ is generated by quadrics. Thus its base locus is exactly $E$. By definition, we have $\{p_1, \dots, p_{10}\}$ contained in the base locus of $H^0(\mathcal{I}_{\{p_1, \dots, p_{10}\}}(2))$. Since $\mathcal{O}_E(p_1+ \cdots + p_{10}) \simeq \mathcal{O}_E(2)$, we have a quadric $Q$ intersecting $E$ at $\{p_1, \dots, p_{10}\}$. Therefore, the base locus of $H^0(\mathcal{I}_{\{p_1, \dots, p_{10}\}}(2))$ is contained in $\{p_1, \dots, p_{10}\}$. In particular, the base locus of $H^0(\mathcal{I}_{\{p_1, \dots, p_{10}\}}(2))$ is exactly $\{p_1, \dots, p_{10}\}$. Iterate Bertini's theorem, we get an open dense subset of the Grassmannian $G(3,H^0(\mathcal{I}_{\{p_1, \dots, p_{10}\}}(2)))$, such that the intersection of each set of three quadrics is smooth away from the base locus $\{p_1, \dots, p_{10}\}$. 
Let $\{Q_i\}_{1 \leq i \leq 5}$ be the generators of $H^0(\mathcal{I}(2))$. We have \[\bigcap_{i=1}^{5} Q_i = E.\] Recall that we have $Q$ such that $Q \cap E = \{p_1, \dots, p_{10}\}$. Take the intersection of these quadrics, we have \[Q \cap \left(\bigcap_{i=1}^{5} Q_i\right) = \Gamma.\] From this, we know that the differentials of all quadrics span the four dimensional cotangent space at every marked point. Thus there exists a triple which has rank three at all ten points. That is to say, there exists an element in $G(3,H^0(\mathcal{I}_{\{p_1, \dots, p_{10}\}}(2)))$, such that the intersection of the corresponding quadrics is smooth at $\{p_1, \dots, p_{10}\}$. Since being smooth at a point is an open condition, there exists an open dense set of $G(3,H^0(\mathcal{I}_{\{p_1, \dots, p_{10}\}}(2)))$, where the intersection of each set of three quadrics is smooth at $\{p_1, \dots, p_{10}\}$. Finally, we conclude that there exist three quadrics $Q_1', Q_2', Q_3'$ (coming from the intersecting of the two open dense subsets), whose intersection is smooth. We may assume their intersection is a complete intersection since being a complete intersection is also an open condition. Then \[C :=\bigcap_{i=1}^{3} Q_i'\] is the desired curve. \\

We show that $p_1, \dots, p_{10}$ fail to impose independent conditions by showing $$H^0(\mathcal{O}_{\mathbb{P}^4}(2)) \rightarrow H^0(\mathcal{O}_{\Gamma}(2)) $$ fails to be surjective. Since $\Gamma \subset E$, it suffices to prove $H^0(\mathcal{O}_{E}(2)) \rightarrow H^0(\mathcal{O}_{\Gamma}(2))$ fails to be surjective. Consider the exact sequence $$0 \rightarrow \mathcal{O}_E(2)(-\Gamma) \simeq \mathcal{O}_E \rightarrow \mathcal{O}_E(2) \rightarrow \mathcal{O}_{\Gamma}(2) \rightarrow 0,$$ which induces long exact sequence in cohomology: $$0 \rightarrow H^0(\mathcal{O}_E) \rightarrow H^0(\mathcal{O}_E(2)) \rightarrow H^0(\mathcal{O}_{\Gamma}(2)) \rightarrow H^1(\mathcal{O}_E) \rightarrow \cdots $$ Since $E$ is proper, geometrically reduced and connected, we have $h^0(\mathcal{O}_E)=1$. Since $h^0(\mathcal{O}_\Gamma(2))=10$, it suffices to show $h^0(\mathcal{O}_E(2))=10$ for each configuration $E$. When $E$ is smooth, with one node, or with a cusp, Riemann-Roch shows that $h^0(\mathcal{O}_E(2))= \deg E + 1 - g(E)$=10. Now let's consider the cases except (1),(2) and (9) in Table \ref{tab:configurations}. For cases where the singularities are all nodes (namely (3)--(8) in Table \ref{tab:configurations}), we consider an example and the rest will be similar. Let's consider the case \\
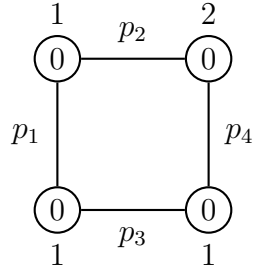
\[\begin{tikzpicture}[
    v/.style={circle, draw, thick, minimum size=6mm, inner sep=0pt}
]
    % Nodes with 0 inside
    \node[v, label=above:$1$] (TL) at (0, 2) {0};
    \node[v, label=above:$2$] (TR) at (2, 2) {0};
    \node[v, label=below:$1$] (BL) at (0, 0) {0};
    \node[v, label=below:$1$] (BR) at (2, 0) {0};

    % Edges
    \draw[thick] (TL) -- (BL) node[midway, left=2pt] {$p_1$};
    \draw[thick] (TL) -- (TR) node[midway, above=2pt] {$p_2$};
    \draw[thick] (BL) -- (BR) node[midway, below=2pt] {$p_3$};
    \draw[thick] (TR) -- (BR) node[midway, right=2pt] {$p_4$};
\end{tikzpicture}\]
where $p_i$ are the corresponding nodes. We have $$ 0 \rightarrow H^0(\mathcal{O}_E(2)) \rightarrow H^0(\mathcal{O}_{\mathbb{P}^1}(2))^{\oplus 3} \oplus H^0(\mathcal{O}_{\mathbb{P}^1}(4)) \xrightarrow{\star} \bigoplus_{i=1}^4 H^0(\mathcal{O}_{p_i}(2)) \rightarrow \cdots$$ We have the map $\star$ to be surjective. In fact, we have $0 \rightarrow \mathcal{O}_{\mathbb{P}^1} \rightarrow \mathcal{O}_{\mathbb{P}^1}(2) \rightarrow \mathcal{O}_{\{p_1,p_2\}}(2) \rightarrow 0$, and $$H^0(\mathcal{O}_{\mathbb{P}^1}(2)) \rightarrow \bigoplus_{i=1}^2 H^0(\mathcal{O}_{p_i}(2))$$ is surjective. Therefore, we conclude that 
\begin{align*}
h^0(\mathcal{O}_E(2)) &= 3h^0(\mathcal{O}_{\mathbb{P}^1}(2)) + h^0(\mathcal{O}_{\mathbb{P}^1}(4)) - h^0\left(\bigoplus_{i=1}^4 H^0(\mathcal{O}_{p_i}(2))\right) \\
&= 9 + 5 - 4 \\
&= 10.
\end{align*}

For the configurations (10)--(15) in Table \ref{tab:configurations}, we have an elliptic $m$-fold point. On the normalizations, we have $10+m$ independent sections. Locally at the singularity, requiring functions to take the same value imposes $m-1$ conditions, and the sum of their derivatives is 0 imposes one more condition. 

Therefore, $h^0(\mathcal{O}_E(2))=10+m-(m-1)-1=10$.\\

So far, we proved that $H^0(\mathcal{O}_{E}(2)) \rightarrow H^0(\mathcal{O}_{\Gamma}(2))$ fails to be surjective. In particular, $H^0(\mathcal{O}_{\mathbb{P}^4}(2)) \rightarrow H^0(\mathcal{O}_{\Gamma}(2)) $ fails to be surjective. \\

This finishes the proof of theorem \ref{thm}.

\section{Finiteness of $A^*(\mathcal{M}')$}\label{Section6}
In this section, we prove that the Chow ring $A^*(\mathcal{M}')$ is finitely generated. 

\begin{defn}
Let $\mathcal{M}_E$ denote the moduli space of tuples 
\[
(E,p_1,\ldots,p_{10},L),
\]
where:
\begin{enumerate}

\item
$L$ is very ample, and the embedding $E\hookrightarrow \bb\bigl(H^0(E,L)^\vee\bigr)$ defined by its complete linear series has one of the configurations in Table~\ref{tab:configurations}.

\item
$p_1,\ldots,p_{10}$ are ordered, distinct points in the
smooth locus of $E$;

\item
$L^{\otimes2}\cong\mathcal O_E(p_1+ \cdots +p_{10}).$
\end{enumerate}    
\end{defn}

Let $\pi:\mathcal{C}_E \rightarrow \mathcal{M}_E$ be the universal curve, with sections $\sigma_i':\mathcal{M}_E \rightarrow \mathcal{C}_E$ for $1 \leq i \leq 10$, and let $\mathcal{L}$ denote the universal line bundle.

Consider the evaluation map
\[
\ev: \Sym^2(\pi_*\mathcal L)\rightarrow\bigoplus_{i=1}^{10}\sigma_i'^*(\mathcal L^{\otimes2}).
\]
From the computation above, we know that the image of $\ev$ is a vector subbundle of rank 9. 
Therefore,
\[
\mathcal V:=\ker(\ev)
\]
is a vector bundle of rank $6$ on $\mathcal M_E$.

By Theorem \ref{thm}, we have \begin{equation}\label{generator}
\mathcal{M''} \subset G(3,\mathcal{V}) \rightarrow \mathcal{M}_E,
\end{equation}
where $\mathcal{M}''$ is moduli space parametrizing tuples $(C,p_1,\dots,p_{10},E)$, where $(C,p_1,\dots,p_{10})\in\mathcal M'$ and $E$ has one of the configurations in Table \ref{tab:configurations} and satisfying
\[
p_1,\ldots,p_{10}\in E_{\mathrm{sm}},
\qquad
\mathcal O_E(p_1+\cdots+p_{10})
\cong \mathcal O_E(2).
\] The forgetful map $\mathcal{M}'' \rightarrow \mathcal{M}'$ is proper and surjective, thus inducing surjective morphism of Chow rings.

\begin{prop}\label{finite}
    The Chow ring $A^*(\mathcal{M}_E)$ is finitely generated. 
\end{prop}

\begin{lemma}\label{sm locus}
    Let $\mathcal{M}_E^{\sm} \subset \mathcal{M}_E$ be the open substack parametrizing tuples $(E,p_1, \cdots, p_{10},L)$ for which $E$ is smooth. Then the Chow ring $A^*(\mathcal{M}_E^{\sm})$ is finitely generated.
\end{lemma}

\textit{Proof}. Let $\mathcal{C}_{1,9}$ be the universal curve over $\mathcal{M}_{1,9}$, with sections $\sigma_1, \dots, \sigma_9$. $\mathcal{C}_{1,9}$ parametrizes tuples $(E, p_1, \dots, p_9,q)$, where $q$ may collide with $p_i$ for some $1 \leq i \leq 9$.
For a tuple in $\mathcal{M}_E^{\sm}$,  $L(-p_1- \cdots -p_4)$ determines a unique point $q$ of $E$. Conversely, we take a tuple $(E, p_1, \dots, p_9,q)$ in $\mathcal{C}_{1,9}$. The relation $L^{\otimes 2}=\mathcal{O}(p_1 + \cdots +p_{10})$ determines a unique point $p_{10}$ satisfying \[
\mathcal O_E(p_{10})\cong\mathcal O_E \bigl(2q+p_1+\cdots+p_4-p_5-\cdots-p_9\bigr).
\] 
Therefore, we have $\mathcal{M}_E^{\sm} \subset \mathcal{C}_{1,9}$ as an open stack. It suffices to prove $A^*(\mathcal{C}_{1,9})$ is finitely generated.
Set $D_i:=\sigma_i(\mathcal M_{1,9}).$
The divisors $D_1,\ldots,D_9$ are disjoint, and
\[
\mathcal C
_{1,9}\setminus\bigcup_{i=1}^{9}D_i
\cong\mathcal M_{1,10}.
\]
We have the excision sequence \[
\bigoplus_{i=1}^{9} A^*(\mathcal M_{1,9})
\xrightarrow{\;\sum_i(\sigma_i)_*\;} A^*(\mathcal C_{1,9}) \longrightarrow A^*(\mathcal M_{1,10}) \longrightarrow 0.\]

Since $A^*(\mathcal{M}_{1,9})=A^*(\mathcal{M}_{1,10})=\mathbb{Q}$ by \cite{B}, we conclude that $A^*(\mathcal{C}_{1,9})$ is generated by classes $\bigl\{[D_i]\bigr\}_{1 \leq i \leq 9}$. In particular, the Chow ring $A^*(\mathcal{C}_{1,9})$ is finitely generated. \qed

Let\[
Z_E=\mathcal M_E\setminus\mathcal M_E^{\sm}.
\]
Stratify $Z_E$ by the configuration types in Table \ref{tab:configurations} and the distribution of the ordered marked points with strata $Z_\alpha$. 

\begin{lemma}\label{nonsm locus}
    The Chow group $A^*(Z_\alpha)$ is trivial for every $\alpha$. Furthermore, the Chow group $A^*(Z_E)$ is generated as a $\mathbb{Q}$-vector space by the fundamental classes of ${Z_\alpha}$.
\end{lemma}
\textit{Proof}.
\textbf{Case 1}. The curve $E$ is one of the nodal configurations (2)-(8) in Table \ref{tab:configurations}. 

Fix the multidegree and the distribution of the ordered markings. Choosing an orientation of its dual cycle gives a degree 2 \'etale cover $q:\widetilde Z_\alpha\longrightarrow Z_\alpha$. Let $r$ be the number of nodes. On each component $\bb^1$, we may fix one of the points to be at 1. The remaining markings among
$p_1,\dots,p_{9-r}$ have coordinates $x_1,\dots,x_{9-r}\in\mathbb G_m$. Since $\Pic^0(E)=\mathbb{G}_m$ by appendix \ref{Pic}, the line bundle $L$ is specified by some $s \in \mathbb{G}_m$. The condition $L^{\otimes2}\simeq\mathcal O_E(p_1+ \cdots + p_{10})$ becomes
\[
s^2=t_{10}\prod_{j=1}^{9-r}x_j.
\]
Thus each piece of $\widetilde Z_E$ with fixed oriented
combinatorial data is an open subset of the graph
$$\Gamma= \left\{(x_1,\ldots,x_{9-r},s,t_{10})\in\mathbb G_m^{11-r}\ \middle|t_{10}=\frac{s^2}{x_1\cdots x_{9-r}}\right\}
\simeq\mathbb G_m^{10-r}.$$
It is an open inclusion because we need all the points to be distinct. Therefore, we conclude that $A^*(\widetilde Z_\alpha)$ is trivial. In particular, $A^*(Z_\alpha)$ is also trivial.\\

\textbf{Case 2}. The curve $E$ is one of the cases (9)--(15) in Table \ref{tab:configurations}, i.e., with an elliptic $m$-fold point.

From appendix \ref{Pic}, we know that $\Pic^0(E)=\mathbb{G}_a$, thus $L$ is uniquely determined after choosing the 10 marked points. Fix the multidegree and the distribution of the ordered markings. On its normalization, place the preimage of the singular point
at $0$. Since each automorphism of $\bb^1$ must preserve $0$ and the direction of the tangent line at $0$, one can assume one marking of each component at $\infty$ and one more marking on one of the components at $1$. The remaining $9-m$ markings then vary in an open subset of $\mathbb G_m^{9-m}$. Therefore, we conclude that $A^*(Z_\alpha)$ is trivial. \qed \\

\textit{Proof of Proposition \ref{finite}}. The conclusion follows from lemma \ref{sm locus}, lemma \ref{nonsm locus}, and the excision sequence $$A^*(Z_E) \rightarrow A^*(\mathcal{M}_E) \rightarrow A^*(\mathcal{M}_E^{\sm}) \rightarrow 0.$$ \qed

\begin{prop}\label{prop6.5}
    The Chow ring $A^*(\mathcal{M}')$ is finitely generated.
\end{prop}
\textit{Proof}. This follows from proposition \ref{finite}, and the composition \eqref{generator}, together with the argument following \eqref{generator}. \qed\\

\textit{Proof of theorem \ref{theorem}}. Let $U_{10}$ be the locus where 10 points impose independent conditions. We have the exact sequences $$A^*(\mathcal{M}_\omega) \rightarrow A^*(\mathcal{M}_{5,10} \setminus \mathcal{M}_{5,10}^3) \rightarrow A^*((\mathcal{M}_{5,10} \setminus \mathcal{M}_{5,10}^3)\setminus\mathcal{M}_\omega) \rightarrow 0,$$
$$A^*( \mathcal{M}') \rightarrow A^*((\mathcal{M}_{5,10} \setminus \mathcal{M}_{5,10}^3)\setminus \mathcal{M}_\omega) \rightarrow A^*(U_{10}) \rightarrow 0.$$
Then theorem \ref{theorem} follows from corollary 2.9 in \cite{Liu}, proposition \ref{prop6.5} and proposition \ref{prop2.1}. \qed

\newpage

\appendix

\section{Equations and specializations of cubic scrolls}

\subsection{Defining equations for the surfaces $S_i$ }
In Section \ref{scrolls}, we prove that each $S_i$ is unique up to projective equivalence for $1 \leq i \leq 6$. In table $\ref{tab:equations}$ we give the equations which cut out each $S_i$.

\begin{table}[htbp]
\centering
\small
\renewcommand{\arraystretch}{1.5}

\caption{Equations of $S_i$ with homogeneous coordinates
$[x:y:z:w:r]$.}
\label{tab:equations}

\begin{tabular}{@{}l@{\qquad}c@{}}
\toprule
Surface & Defining equations \\
\midrule
$S_1$ & $V(xr-wy,\;rz-xw,\;x^2-yz)$ \\
$S_2$ & $V(xr-wy,\;x^2-yr,\;r^2-wx)$ \\
$S_3$ & $V(yr-xz,\;wy,\;wx)$ \\
$S_4$ & $V(yr-x^2,\;wy,\;wx)$ \\
$S_5$ & $V(yr,\;wy,\;wx)$ \\
$S_6$ & $V(xy,\;wy,\;wx)$ \\
\bottomrule
\end{tabular}
\end{table}

\subsection{Flat families realizing the specializations}\label{A.2}

In Table \ref{tab:cubic-surface-specializations}, each row defines a one-parameter family. This is a flat family because each fiber has the same Hilbert polynomial $P(n)=\frac{3n^2+5n+2}{2}$. For the specialization $S_i \rightsquigarrow S_j$, the general fiber is $S_i$ and the central fiber is $S_j$.

\begin{table}[htbp]
\centering
\small
\renewcommand{\arraystretch}{1.6}

\caption{Explicit specializations.}
\label{tab:cubic-surface-specializations}

\begin{tabular}{@{}l@{\qquad}c@{}}
\toprule
Specialization & The corresponding specialization \\
\midrule

$S_1\rightsquigarrow S_2$
&
$\mathcal{S}_t
 =V\bigl(xr-wy,\;x^2-y(r+tz),\;r(r+tz)-wx\bigr)$
\\

$S_1\rightsquigarrow S_3$
&
$\mathcal{S}_t
 =V\bigl(tr^2-xw,\;xz-yr,\;yw-trz\bigr)$
\\

$S_2\rightsquigarrow S_4$
&
$\mathcal{S}_t
 =V\bigl(txr-wy,\;x^2-yr,\;tr^2-wx\bigr)$
\\

$S_3\rightsquigarrow S_4$
&
$\mathcal{S}_t
 =V\bigl(yr-x(x+tz),\;wy,\;wx\bigr)$
\\

$S_4\rightsquigarrow S_5$
&
$\mathcal{S}_t
 =V\bigl(yr-tx^2,\;wy,\;wx\bigr)$
\\

$S_5\rightsquigarrow S_6$
&
$\mathcal{S}_t
 =V\bigl(y(x+tr),\;wy,\;wx\bigr)$
\\

\bottomrule
\end{tabular}
\end{table}

\section{local ring description of configurations (9)--(15)}

In configurations (9)--(15) of Table \ref{tab:configurations}, the curve $E$ only has one singular point, and its normalization is 
\[
\nu \colon \widetilde E
   = \coprod_{i=1}^{m} \mathbb P^1_i
   \longrightarrow E.
\]

Choose coordinates $t_i$ on each normalization component $\bb^1_i$, such that $t_i(p_i)=0$. After suitable rescaling of these coordinates, the local ring of E at the singularity is given by 

\[
R_m :=
\left\{
(f_1,\ldots,f_m)\in\bigoplus_{i=1}^{m} k[t_i]_{(t_i)}
\ \middle|\
\begin{gathered}
f_1(0)=\cdots=f_m(0),\\
\sum_{i=1}^{m} f_i'(0)=0
\end{gathered}
\right\}.
\]

The conditions on the value of each $f_i$ at 0 impose $m-1$ independent linear conditions, and the derivative relation imposes
one more condition. Thus $\delta_p=m$.

Therefore, we have
\[
p_a(E)
 = \sum_{i=1}^{m} g(\mathbb P^1_i)- m + 1 + \delta_p 
 = 1.
\]

Therefore, the configurations (9)--(15) of Table \ref{tab:configurations} are all genus 1 curves of degree 5.

\begin{remark}
From the local ring description of configurations (9)--(15) in Table \ref{tab:configurations}, they are elliptic $m$-fold points according to Lemma 2.2 in \cite{S}.  
\end{remark}

\section{Uniqueness for each configuration of (2)--(15) in Table \ref{tab:configurations}}\label{Pic}
\subsection{$\Pic^d(E)$ for each configuration}

\begin{lemma}
Let $E$ be one of configurations (2)--(15). Fix a multidegree
$\mathbf d=(d_1,\ldots,d_r)$ on its normalization
\[
\nu:\widetilde E=\coprod_{i=1}^r\mathbb P^1_i\longrightarrow E.
\]
Then we have
\[
\operatorname{Pic}^{\mathbf d}(E)\simeq
\begin{cases}
\mathbb G_m, & \text{in configurations (2)--(8)},\\
\mathbb G_a, & \text{in configurations (9)--(15)}.
\end{cases}
\]
\end{lemma}

\begin{proof}
Now we fix a multidegree $\mathbf{d}$ and we classify the ways the line bundles $\mathcal O_{\mathbb P^1_i}(d_i)$ descend to $E$.\\

\textit{The nodal case.}
A descent is specified by nonzero gluing scalars $(\alpha_1,\ldots,\alpha_r)$ at all the nodes. Rescaling the line bundle on each component changes this tuple by an element of 
\[
H=\left\{(a_1,\ldots,a_r)\in(\mathbb G_m)^r:
                    \prod_{i=1}^r a_i=1\right\}.
\]
Therefore,
\[
\operatorname{Pic}^{\mathbf d}(E)
 \simeq (\mathbb G_m)^r/H
 \xrightarrow{\ \sim\ }\mathbb G_m,
\quad \text{given by the map }
\bigl[(\alpha_i)\bigr]\longmapsto\prod_{i=1}^r\alpha_i.
\]

\textit{The elliptic $m$-fold case.}
Since the singularity is an elliptic $m$-fold point, we have $r=m$. Choose the coordinate $t_i$ on each component such that the functions descend when their values agree and their first derivatives sum to zero. Since gluing constraints are at order 0 and 1, the gluing data are determined by units in
\[A_m=\bigoplus_{i=1}^m k[\varepsilon_i]/(\varepsilon_i^2).\]
Set \[H_m=\left\{(a+b_i\varepsilon_i)_{i=1}^m:a\in k^\times,\ \sum_{i=1}^m b_i=0\right\}\subset A_m^\times.\]
The gluing data are represented by units $\gamma_i+\beta_i\varepsilon_i$, where $\gamma_i\in k^\times$. Rescaling independently on
the normalization components, we may normalize $\gamma_i=1$ and write the gluing units as $1+\beta_i\varepsilon_i$.
Therefore, 
\[
\Pic^d(E)=
\frac{A_m^\times}{(k^\times)^m\cdot H_m}
\simeq \mathbb G_a, \quad \text{given by the map }
\bigl[(\beta_i)\bigr]\longmapsto\sum_{i=1}^m\beta_i.
\]
where $(k^\times)^m$  accounts for independent rescalings on each component, and $H_m$ accounts for changes of local trivialization at the singularity.
\end{proof}
\subsection{Uniqueness of configurations (2)--(15).}

We first consider the nodal configurations (2)--(8). Fix the multidegree $\mathbf d=(d_1,\ldots,d_r)$, where $d_i>0$ and $\sum_i d_i=5$. The normalization is $\nu:\coprod_{i=1}^r\mathbb P^1_i\longrightarrow E.$ Choose coordinates $t_i$ so that the preimages of the nodes are $0_i$ and $\infty_i$, with $\infty_i$ identified with $0_{i+1}$, cyclically. After fixing $0$ and $\infty$ on each component $\bb^1$ (in the normalization), its automorphism can only be the scaling map $t_i \rightarrow \lambda_it_i$. Recall that the gluing is described by scalars $\alpha_i \in k^\times$, which corresponds to the element \[
q=\prod_{i=1}^r\alpha_i,
\]
under the identification $\Pic^\mathbf{d} \simeq \mathbb{G}_m$. Therefore, this automorphism acts on $\Pic^\mathbf{d}$ by:  \[
q\longmapsto \left(\prod_{i=1}^r\lambda_i^{d_i}\right)q.\]

Since our base field $k$ is assumed to be algebraically closed, we conclude that this action is transitive. Therefore, each nodal configurations of (2)--(8) is unique up to automorphism of $\bb^4$. 

Now we consider the configurations (9)--(15). Fix the multidegree $\mathbf d=(d_1,\ldots,d_m)$, where $d_i>0$ and $\sum_i d_i=5$. The normalization is $\nu:\coprod_{i=1}^m\mathbb P^1_i\longrightarrow E.$ Choose coordinates $t_i$ so that the preimage of the singularity is $\infty$. On each component, consider the automorphisms restricting to identity on $2\infty$. Any such automorphism can only be the translation $t_i \rightarrow t_i+\tau_i$. Recall that the gluing data is described by $\{1+\beta_i\varepsilon_i\}_i \in A_m^\times$ which corresponds to the element \[
q=\sum_{i=1}^m\beta_i,
\]
under the identification $\Pic^\mathbf{d} \simeq \mathbb{G}_a$. Therefore, this automorphism acts on $\Pic^\mathbf{d}$ by:  \[
q\longmapsto q+\left(\sum_{i=1}^m d_i\tau_i\right).\]

We conclude that this action is transitive. Therefore, each elliptic $m$-fold configuration of (9)--(15) is unique up to automorphism of $\bb^4$.

\section{Equations for configurations (2)--(15)}\label{quadrics}
In this section, we write down the explicitly equations cutting out the configurations (2)--(15). Denote the homogeneous coordinates of $\bb^4$ by $[x:y:z:w:r]$.
\begingroup
\small
\renewcommand{\arraystretch}{1.7}
\setlength{\LTleft}{\fill}
\setlength{\LTright}{\fill}
\setlength{\LTcapwidth}{\textwidth}

\begin{longtable}{@{}c@{\qquad}c@{}}
\caption{Equations for configurations (2)--(15) in Table \ref{tab:configurations}.}
\label{tab:configuration-equations}\\

\toprule
Configuration & Defining equations \\
\midrule
\endfirsthead

\toprule
Configuration & Defining equations \\
\midrule
\endhead

\bottomrule
\endfoot

$(2)$ &
$V\!\left(
r^2-xw,\ xz-ry,\ yw-rz,\
y^2-rw-r^2,\ yz-w^2-rw
\right)$ \\

$(3)$ &
$V\!\left(y^2-xz,\ yz-xw,\ z^2-yw,\ zw-yr,\ w^2-zr\right)$ \\

$(4)$ &
$V\!\left(yw,\ yr,\ xz-y^2-wr,\ xr-w^2,\ zw-r^2\right)$ \\

$(5)$ &
$V\!\left(xw,\ yr,\ wr,\ yz-w^2,\ xz-r^2\right)$ \\

$(6)$ &
$V\!\left(yw,\ yr,\ xr-w^2,\ zw-r^2,\ xz-wr\right)$ \\

$(7)$ &
$V\!\left(xz,\ yw,\ yr,\ zr,\ xw-r^2\right)$ \\

$(8)$ &
$V\!\left(xz,\ xw,\ yw,\ yr,\ zr\right)$ \\

$(9)$ &
$V\!\left(r^2-xw,\ xz-ry,\ yw-rz,\ y^2-rw,\ yz-w^2\right)$ \\

$(10)$ &
$V\!\left(yz-xw,\ yw-xr,\ z^2-xr,\ zw-yr,\ w^2-zr\right)$ \\

$(11)$ &
$V\!\left(y^2-xz-xw,\ yw-xr,\ w^2-yr,\ zw,\ zr\right)$ \\

$(12)$ &
$V\!\left(yz-xw,\ yw-xr,\ zw-xr,\ yr-zr,\ w^2-yr\right)$ \\

$(13)$ &
$V\!\left(yz-xr,\ z^2-yz-xw,\ yw,\ yr-zr,\ wr\right)$ \\

$(14)$ &
$V\!\left(yz-xr,\ yw-xr,\ zw-xr,\ yr-zr,\ yr-wr\right)$ \\

$(15)$ &
$V\!\left(yz-yw,\ yz-yr,\ yz-zw,\ yz-zr,\ yz-wr\right)$ \\

\end{longtable}
\endgroup

In particular, we prove that the ideal sheaf of each configuration of (2)--(15) is generated by five quadrics. By \cite{SD}, the ideal sheaf of configuration (1) is also generated by five quadrics.


\begin{thebibliography}{10}

\bibitem{Coskun}
I. Coskun, \textit{Degenerations of surface scrolls and the Gromov--Witten invariants of Grassmannians}, J. Algebraic Geom. \textbf{15} (2006), no. 2, 223--284. MR 2199064

\bibitem{Liu}
Y. Liu, \textit{On the Chow rings of the moduli spaces $\mathcal{M}_{5,8}$ and $\mathcal{M}_{5,9}$}, arXiv:2509.02950.

\bibitem{B}
P. Belorousski, \textit{Chow rings of moduli spaces of pointed elliptic curves}, Ph.D. thesis, University of Chicago, 1998.

\bibitem{M4}
C. Faber, \textit{Chow rings of moduli spaces of curves. II. Some results on the Chow ring of $\overline{\mathcal{M}}_4$}, Ann. of Math. (2) \textbf{132} (1990), no. 3, 421–449. MR 1078265

\bibitem{M5}
E. Izadi, \textit{The Chow ring of the moduli space of curves of genus 5}, The moduli space of curves (Texel Island, 1994), Progr. Math., vol. 129, Birkhäuser Boston, Boston, MA, 1995, pp. 267–304. MR 1363060

\bibitem{M6}
N. Penev and R. Vakil, \textit{The Chow ring of the moduli space of curves of genus six},  Algebr. Geom.
\textbf{2} (2015), no. 1, 123–136. MR 3322200

\bibitem{M789}
S. Canning and H. Larson, \textit{The Chow rings of the moduli spaces of curves of genus 7, 8, and 9}, J. of Algebraic Geom. \textbf{33} (2024), no.1, 55–116, MR4693574

\bibitem{HS}
S. Canning and H. Larson, \textit{On the Chow and cohomology rings of moduli spaces of stable curves}, J. Eur. Math. Soc. \textbf{28} (2026), no.~9, 3869--3918. 


\bibitem{S}
D.~I. Smyth, \textit{Modular compactifications of the space of pointed elliptic curves I}, Compos. Math. \textbf{147} (2011), no.~3, 877--913. MR~2801404.

\bibitem{SD}
B. Saint-Donat, \textit{On Petri's analysis of the linear system of quadrics through a canonical curve}, Math. Ann. \textbf{206} (1973), 157--175.

\end{thebibliography}
\end{document}